\documentclass[11pt,reqno]{amsart}
\usepackage{filecontents}
\usepackage{microtype}
\usepackage[english]{babel}
\usepackage[T1]{fontenc}
\usepackage[utf8]{inputenc}
\usepackage[T1]{fontenc}
\usepackage{lmodern}
\usepackage{latexsym}
\usepackage{verbatim}
\usepackage{hyperref}
\usepackage{amsthm}
\usepackage{amsmath}
\usepackage{amsfonts}
\usepackage[italian]{varioref}
\usepackage{imakeidx}
\usepackage{amssymb}
\usepackage{xcolor}
\usepackage{esint}
\usepackage{accents}
\usepackage{enumitem}
\usepackage{caption}
\usepackage[justification=centering]{caption}

\usepackage{tikz}
\usepackage{tikz,tikz-3dplot}
\usepackage{pgfplots}
\pgfplotsset{compat=1.16}
\tikzset{
    cross/.pic = {
    \draw[rotate = 45] (-#1,0) -- (#1,0);
    \draw[rotate = 45] (0,-#1) -- (0, #1);
    }
}
\usetikzlibrary{calc, angles}
\usetikzlibrary{arrows}
\usetikzlibrary{positioning}
\usetikzlibrary{intersections}
\usepackage{pgfplots}
\pgfplotsset{compat=newest}
\usepgfplotslibrary{fillbetween}
\usepackage{pgfplots}
\pgfplotsset{compat=newest}
\usepgfplotslibrary{fillbetween}
\usetikzlibrary{patterns}
\usepackage{mathtools}
\mathtoolsset{showonlyrefs}
\allowdisplaybreaks

\definecolor{ao(english)}{rgb}{0.0, 0.5, 0.0}

\usepackage{lipsum}
\usepackage{curve2e}
\definecolor{gray}{gray}{0.4}

\usepackage{amsmath}
\usepackage{amsfonts}
\usepackage{amssymb}
\usepackage{graphicx}
\usepackage{hyperref} 
\usepackage{mathrsfs} 
\usepackage{relsize} 
\usepackage{float}

\theoremstyle{plain}
\newtheorem{theorem}{Theorem}[section]

\newtheorem{definition}[theorem]{Definition}
\newtheorem{lemma}[theorem]{Lemma}
\newtheorem*{problem*}{Problem}

\newtheorem{proposition}[theorem]{Proposition}

\newtheorem{theoremA}{Theorem}

\newtheorem*{theorem*}{Theorem}

\theoremstyle{definition}

\newtheorem{remark}[theorem]{Remark}

\def\Item$#1${\item $\displaystyle#1$
   \hfill\refstepcounter{equation}(\theequation)}
   
\theoremstyle{plain} 
\newtheorem{thm}{Theorem}[section] 
\newtheorem{cor}[thm]{Corollary}

\theoremstyle{definition}

\theoremstyle{remark} 
\numberwithin{equation}{section}
\usepackage[makeroom]{cancel}

\usepackage{bm} 
\usepackage{graphicx}
\makeindex
\usepackage{mathtools}
\usepackage{slashed}
\date{}

\usepackage{tikz,tikz-3dplot}
\tdplotsetmaincoords{80}{45}
\tdplotsetrotatedcoords{-90}{180}{-90}
\usetikzlibrary{arrows, automata, intersections}

\newcommand{\rr}{\mathbb{R}}
\newcommand{\hh}{\mathbb{H}}

\newcommand{\sss}{\mathbb{S}}
\newcommand{\e}{\epsilon}

\newcommand{\inn}{\textnormal{in}\ }
\newcommand{\onn}{\textnormal{on}\ }
\newcommand{\andd}{\textnormal{and}}

\newcommand{\dint}[1]{\ \textnormal{d}#1}

\newcommand{\dvol}{\ \textnormal{dv}}
\newcommand{\dsup}{\ \textnormal{da}}
\newcommand{\tr}[1]{\textnormal{tr}\left( #1 \right)}
\newcommand{\heat}{\left( \frac{\partial}{\partial t}+\Delta\right)}
\newcommand{\isom}[1]{\textnormal{Isom}\left(#1 \right)}
\newcommand{\derive}[2]{\frac{\partial #1}{\partial #2}}
\newcommand{\deriven}[3]{\dfrac{\partial^{#1}#2}{{\partial#3}^{#1}}}
\newcommand{\scal}[2]{\langle{#1},{#2}\rangle}
\newcommand{\twosystem}[2]{\left\{\begin{aligned} &#1\\ &#2\end{aligned}\right.}

\newcommand{\bd}{\partial}

\newcommand{\abs}[1]{\lvert{#1}\rvert}

\newcommand{\matrice}{\begin{pmatrix}}
\newcommand{\ok}{\end{pmatrix}}

\newcommand{\twomatrix}[4]{\begin{pmatrix} #1&#2\\#3&#4\end{pmatrix}}

\usepackage{tikz}
\usepackage{pgfplots}
\pgfplotsset{compat=1.18}

\tikzstyle{mybox} = [draw=black, very thick, rectangle, rounded corners, inner ysep=5pt, inner xsep=5pt]
\usepackage{pgfplots}
\pgfplotsset{compat=1.18}

\title[]{Rigidity of an overdetermined heat equation and minimal helicoids in space-forms}
\author{Andrea Bisterzo and Alessandro Savo}

\address{Andrea Bisterzo
  \newline \indent Centro di Ricerca Matematica Ennio De Giorgi
  \newline \indent
  Scuola Normale Superiore di Pisa
\newline\indent Piazza dei Cavalieri 3, 56126, Pisa, Italy.}
\email{andrea.bisterzo@sns.it}

\address{Alessandro Savo
  \newline \indent Dipartimento di Scienze di Base e Applicate per l’Ingegneria
  \newline \indent
Sapienza Università di Roma 
\newline\indent Via Scarpa 16, 00161,
Roma, Italy.}
\email{alessandro.savo@uniroma1.it}

\begin{document}

\begin{abstract} Let $M$ be a Riemannian manifold and $\Omega$ a smooth domain of $M$. We study the following heat diffusion problem: assume that the initial temperature is equal to $1$, uniformly on $\Omega$, and is $0$ on its complement. Heat will then flow away  from $\Omega$  to its complement, and we are interested in the temperature on the boundary of $\Omega$ at all positive times $t>0$.  In particular we ask: are there domains for which the temperature at the boundary is a constant $c$,  for all positive times $t$ and for all points of the boundary?  If they exist, what can we say about their geometry?

This is a typical example of overdetermined heat equation. It is readily seen that if $c$ exists it must be $\frac 12$, and domains with constant boundary temperature will be said to have the $\frac 12$-property.  Previous work by \cite{MPS06} and \cite{CSU23} show that, on $\mathbb R^3$, the only such domains (up to congruences) have boundary which is a plane or (a bit surprisingly) the right helicoid. 


In this paper we first show that, in great generality, the boundary of a $\frac 12$-domain must be minimal; we then extend (with a different proof) the above classification from $\mathbb R^3$ to the other $3$-dimensional space-forms.  We prove that, in $\mathbb S^3$, $\frac 12$-domains are bounded by a totally geodesic surface or the Clifford torus, and in the hyperbolic space $\mathbb H^3$ are bounded by a totally geodesic surface or by an (embedded) minimal hyperbolic helicoid. 
As a by-product, we extend (with a different proof) a result by Nitsche on uniformly dense domains from $\mathbb R^3$ to $3$-dimensional space-forms. 
\end{abstract}

\maketitle

\section{Introduction}\label{introduction}

\subsection{The Cauchy temperature and the main result} Let $M$ be a smooth, stochastically complete Riemannian manifold and $\Omega\subset M$ a possibly unbounded smooth domain. We assume that $M$ has empty boundary. The Riemannian metric on $M$ (and on any of its hypersurfaces) will be denoted by $\scal{\cdot}{\cdot}$.

The \textit{Cauchy temperature $u_C$ associated to $\Omega$} is defined as the unique  bounded solution $u_C=u_C(t,x)$ on $(0,\infty)\times M$ to the following initial value problem:
\begin{align}\label{Prob:Cauchy}
\begin{cases}
\heat u_C=0 & \inn (0,+\infty)\times M\\
u_C=\chi_{\Omega} & \onn \{0\}\times M,
\end{cases}
\end{align}
where $\Delta$ denotes the positive definite Laplace-Beltrami operator associated to $(M,\scal{\cdot}{\cdot})$ (on $\rr^n$ it is $\Delta u=-\sum_j u_{jj}$) and $\chi_\Omega$ is the characteristic function of the domain $\Omega$. The assumption that $M$ is stochastically complete is needed to ensure that \eqref{Prob:Cauchy} admits at most one bounded solution (see \cite[Theorem 6.2]{Gr99} and the literature cited therein). Sufficient conditions for stochastic completeness are found in \cite{Gr99}; a classical fact is that every geodesically complete manifold with Ricci curvature bounded from below is stochastically complete \cite{Yau78} (see also \cite{Hs02, Gr09}).

In this paper we are interested in the behavior of the temperature $u_C$ on the boundary of $\Omega$. It can be shown that
$\lim_{t\to 0^+}u_C(t,x)=\frac 12$ for all $x\in\bd\Omega.$ There are situations where, however, the boundary temperature is $\frac 12$ for all times: 
\begin{equation}\label{half}
u_C(t,x)=\frac 12 \quad\text{for all $t>0$ and $x\in\bd\Omega$.}
\end{equation}
The overdetermination \eqref{half} rarely holds; it holds, for example, when $\Omega$ is a half-space in $\mathbb R^n$ or a hemisphere in $\mathbb S^n$; more generally, \eqref{half} holds whenever there is an involutive isometry taking $\Omega$ to the interior of its complement and fixing the boundary of $\Omega$.
More general sufficient conditions for having \eqref{half} are given in Section \ref{Sec:SufficientConditions} below.

\begin{definition}
    We say that $\Omega$ has the \textnormal{$\frac 12$-property}, or that is a \textnormal{$\frac 12$-domain}, if the Cauchy temperature on $\Omega$ satisfies \eqref{half}.
\end{definition} 

\noindent In other words, $\frac 12$-domains are precisely those domains which support a solution to the overdetermined heat equation:
\begin{align}\label{overhalf}
\begin{cases}
\heat u_C=0 & \inn (0,+\infty)\times M\\
u_C=\chi_{\Omega} & \onn \{0\}\times M\\
u_C=\frac 12 & \onn (0,\infty)\times\bd\Omega.
\end{cases}
\end{align}

\noindent{\bf Assumptions on $\Omega$.} In what follows we assume that $\Omega$ is a (possibly unbounded) domain with smooth boundary $\Sigma$ in a stochastically complete Riemannian manifold $M$, such that
$
\bd\Omega=\bd(M\setminus\overline\Omega)\doteq\Sigma.
$
In particular, setting $\Omega_+=\Omega$ and $\Omega_-=M\setminus\overline{\Omega}$, we have a disjoint union
\begin{equation}\label{disunion}
M=\Omega_+\sqcup\Sigma\sqcup\Omega_-.
\end{equation}
$\Sigma$ is then a  smooth, complete $(n-1)$-dimensional hypersurface of $M$. We also assume that $\Sigma$ is connected.

\medskip
In this paper we first prove that the $\frac 12$-property implies minimality, under rather general assumptions:

\begin{theoremA}\label{Thm_MinimalityBoundary}
Let $\Omega$ be a (possibly unbounded) domain in a stochastically and geodesically complete $n$-dimensional Riemannian manifold $M$. If $\Omega$ has the $\frac{1}{2}$-property, then $\partial\Omega$ is minimal. 
\end{theoremA}

When $M=\rr^n$ this fact has been proved in \cite{CSU23} by a different method, using the explicit expression of the Euclidean heat kernel.  In the same paper, using a  result by Nitsche,  the author study the rigidity of the $\frac 12$-property in $\rr^3$, showing that

\begin{theorem}[Theorem 1.2 in \cite{CSU23}]\label{MPS} $\Omega\subseteq \rr^3$ has the $\frac 12$-property if and only is its boundary is either a plane or the right minimal helicoid. 
\end{theorem}
\noindent We will say few words on the proof in Subsection \ref{uniformdensity}. \medskip

The main result of this paper is to give a classification of $\frac 12$-domains in $3$-dimensional space forms, by using an approach different to the one in \cite{CSU23} (we will sketch our argument in Section \ref{Subsec:Approach} below). 

In this paper, $M^3_{\sigma}$ will denote the $3$-dimensional manifold of constant sectional curvature $\sigma$ (space-form). It is enough to consider the cases where $\sigma=0$ ($M=\mathbb R^3$), $\sigma=1$ ($M=\mathbb S^3$) and $\sigma=-1$, in which case  $M=\mathbb H^3$, the hyperbolic $3$-space. 

\begin{theoremA}\label{Thm_Main}
Let $\Omega$ be a smooth domain  in the $3$-dimensional space-form $M_{\sigma}$. Then, $\Omega$ has the $\frac 12$-property if and only if its boundary is totally geodesic or:

1) the right helicoid if $M=\rr^3$;

2) the Clifford torus if $M=\sss^3$;

3) an embedded minimal hyperbolic helicoid if $M=\hh^3$.

\end{theoremA}

The first two surfaces are well-known (parametrizations will be given in Section \ref{Sec:SufficientConditions}). Let us describe the family of  embedded minimal hyperbolic helicoids, which is infinite, and is  indexed by $\alpha\in (0,\infty)$. These surfaces have been studied in \cite{dCD83, Mor81, Mor82, Wa19}. A parametrization in the
hyperboloid model of $\hh^3$ is given as follows:
\begin{align}
X^\alpha:\rr^2 &\to \hh^3\\
(u,v) &\mapsto \left(\begin{array}{c} \cosh(\alpha u) \cosh(v)\\ \sinh(\alpha u) \cosh(v)\\ \cos(u) \sinh(v)\\ \sin(u) \sinh(v) \end{array} \right).
\end{align}
In the Poincaré ball model of $\hh^3$ the parametrization becomes
\begin{align}
X^\alpha:\rr^2 & \to \hh^3\\
(u,v) &\mapsto \left(\begin{array}{c} \frac{\sinh(\alpha u) \cosh(v)}{1+\cosh(\alpha u) \cosh(v)}\\ \frac{\cos(u) \sinh(v)}{1+\cosh(\alpha u) \cosh(v)} \\ \frac{\sin(u) \sinh(v)}{1+\cosh(\alpha u) \cosh(v)}\end{array} \right).
\end{align}
\noindent In Figure \ref{Fig:HyperbolicHelicoids}, the plots of some embedded minimal hyperbolic helicoid in the Poincaré ball.

\begin{figure}[h!]
\centering
\includegraphics[width=3cm]{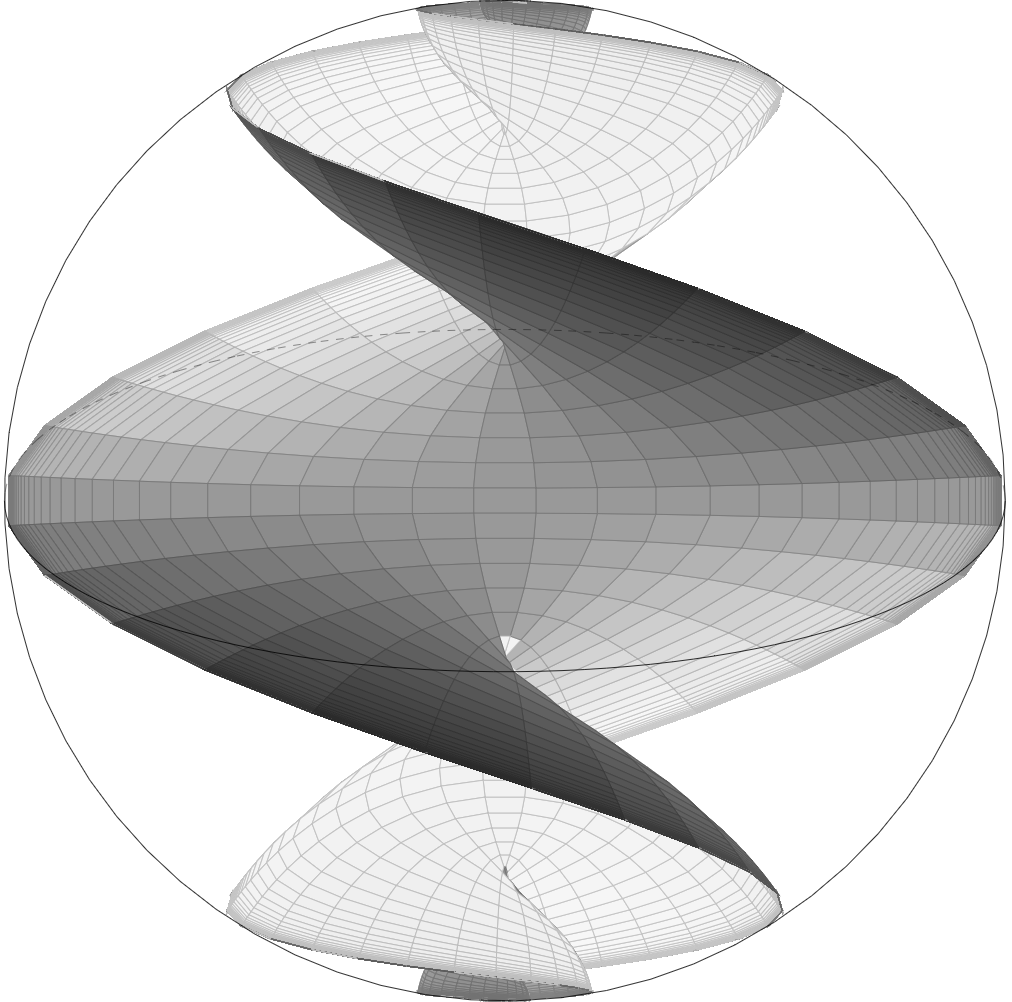}\hspace{1.6cm}
\includegraphics[width=3cm]{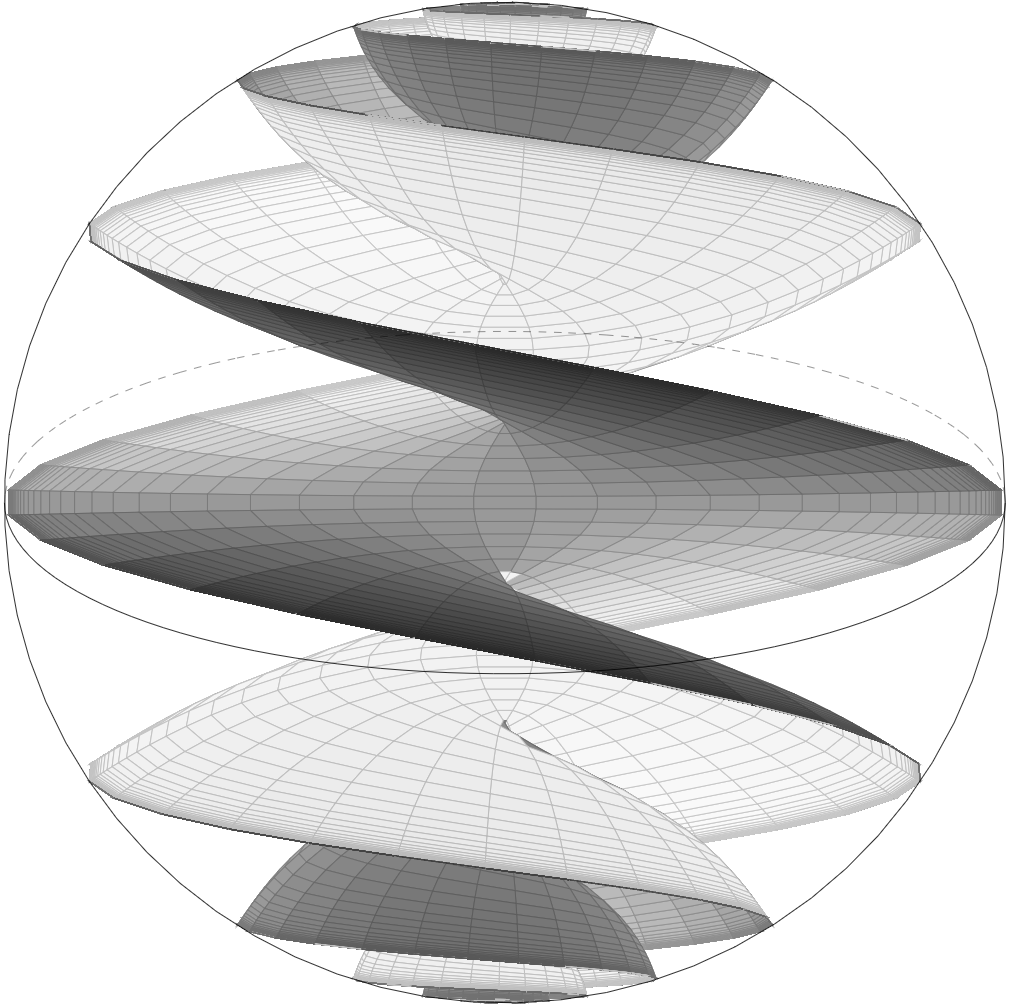}\hspace{1.6cm}
\includegraphics[width=3cm]{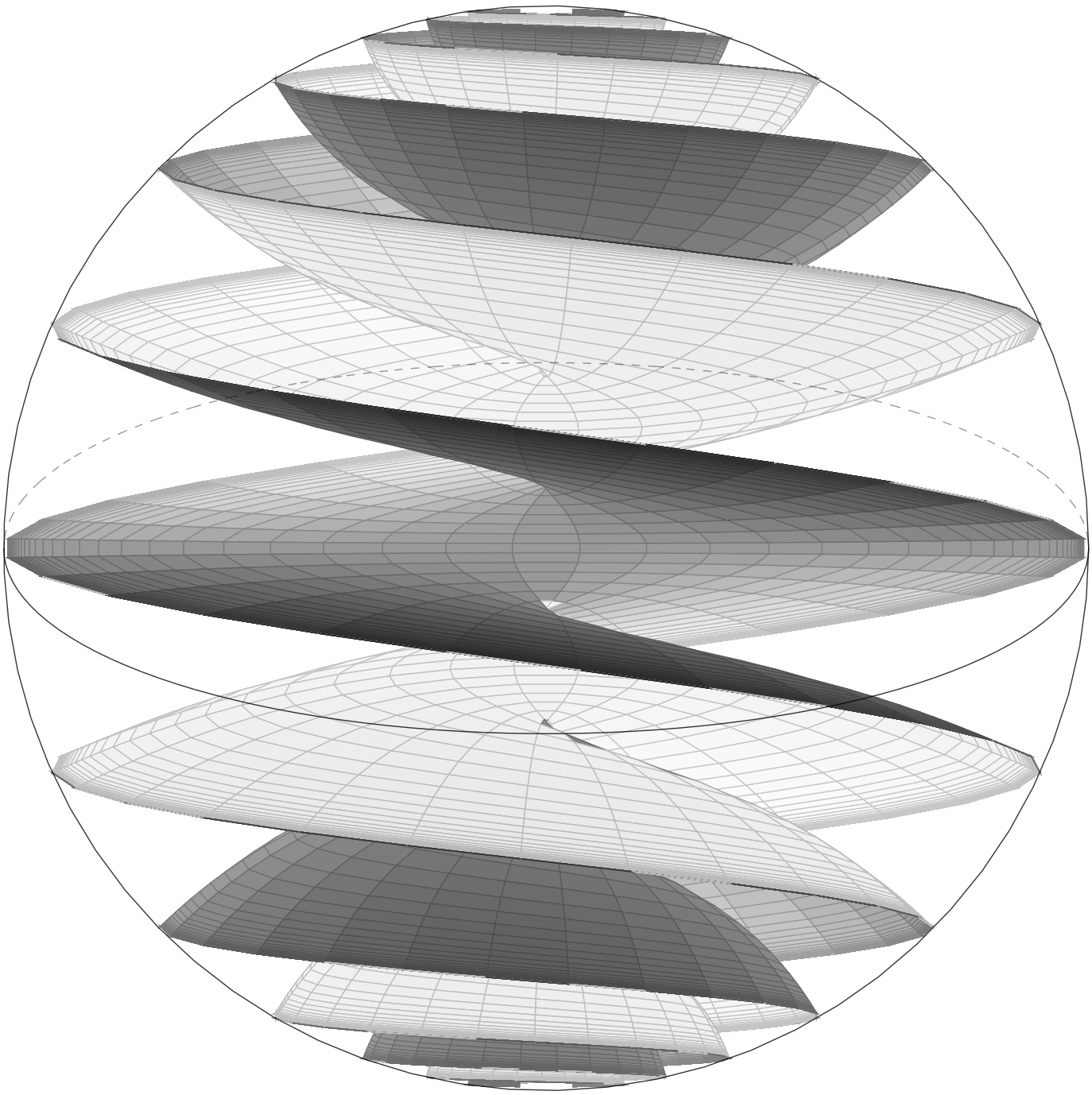}\\

\caption{Embedded minimal hyperbolic helicoids in the Poincaré ball model for $\alpha=\frac{2}{5}, \frac{1}{5}, \frac{7}{50}$.}
\label{Fig:HyperbolicHelicoids}
\end{figure}

\begin{remark}  In this paper, the term {\it embedded minimal helicoid} will refer to one of the surfaces 1), 2), 3) listed in Theorem \ref{Thm_Main}. Indeed, also the Clifford torus can be seen as a (minimal) helicoidal surface in the $3$-sphere (see Remark \ref{Rmk_SphericalHelicoids}).
\end{remark}

\subsection{Uniform density and the argument in \cite{CSU23}}\label{uniformdensity} As explained in \cite{CSU23, MPS06}, when $\Omega\subseteq\mathbb R^n$, the  $\frac 12$-property is equivalent to the following property introduced by Cimmino in \cite{Ci}: the volume
of the intersection of $\Omega$ with a ball in $\mathbb R^n$ centered at a point of the boundary is always half the volume of the ball. There is an equivalent formulation: for $r>0$ and $x\in\bd\Omega$ introduce the {\it density function} $\sigma: (0,\infty)\times\bd\Omega\to \mathbb R$ defined by:
$$
\sigma(r,x)=\frac{\abs{\Omega\cap\bd B_r(x)}}{\abs{\bd B_r(x)}};
$$
then, following \cite{MPS06}, we say that $\Omega$ is {\it $\frac 12$-uniformly dense} in $\bd\Omega$ if 
\begin{equation}\label{densitycondition}
\sigma(r,x)=\frac 12
\end{equation}
for almost every $r>0$ and for every $x\in\bd\Omega$.

Using the explicit representation of the heat kernel of $\mathbb R^n$, we have for any $(t,x)\in (0,+\infty)\times \rr^n$
$$
u_C(t,x)=\dfrac{1}{(4\pi t)^{n/2}}\int_{\Omega}e^{-\frac{\abs{x-y}^2}{4t}}\,\dvol_y.
$$
Fixed $x\in\bd\Omega$, integrating on concentric spheres centered at $x$ one obtains, by the co-area formula,
\begin{equation}\label{cauchye}
\begin{aligned}
u_C(t,x)&=\dfrac{1}{(4\pi t)^{n/2}}\int_0^{\infty}e^{-\frac{r^2}{4t}}\abs{\Omega\cap \bd B_r(x)}\,dr\\
&=\dfrac{1}{(4\pi t)^{n/2}}\int_0^{\infty}e^{-\frac{r^2}{4t}}\abs{\bd B_r(x)}\sigma(r,x)\,dr
\end{aligned}
\end{equation}
and, using the injectivity of the Laplace transform, one sees that  $u_C(t,x)=\frac 12$ for all $(t,x)\in (0,+\infty)\times \partial \Omega$  if and only if $\sigma(r,x)=\frac 12$ for all $x\in \partial \Omega$ and for almost every $r>0$ (\cite[Page 14]{CSU23}). Thus, 
\begin{center}
    {\it For domains in $\mathbb R^n$, the $\frac 12$-property is equivalent\\ to $\Omega$ being $\frac 12$-uniformly dense in $\bd\Omega$.}
\end{center}

\smallskip

Hence, in discussing the $\frac 12$-property, it is enough to study the equivalent condition \eqref{densitycondition}. Now the function $\sigma(r,x)$ admits, as $r\to 0$, a Taylor expansion of type:
\begin{equation}\label{cexpansion}
\sigma(r,x)\sim\frac 12+\sum_{k=0}^{\infty}\sigma^{(k)}(x)r^{k+1}
\end{equation}
for certain invariants $\sigma^{(k)}(x)$ depending on $x$. Thus, $\frac 12$-uniform density \eqref{densitycondition} implies that all such invariants vanish.  It is proved in \cite{MPS06} that $\sigma_0(x)$ is a non-zero multiple of the mean curvature $H$ of $\bd\Omega$ at $x$, which implies that $\bd\Omega$ must be minimal (in $\mathbb R^3$, this fact was first observed by Cimmino in \cite{Ci}). So we get

\begin{theorem}[Theorem 1.2 in \cite{CSU23}]
Domains in $\mathbb R^n$ with the $\frac 12$-property have minimal boundary.    
\end{theorem}

\smallskip

We remark that minimality is a necessary condition, but definitely not a sufficient condition for the $\frac 12$-property; to get such, one restricts to domains in $\mathbb R^3$ and examines higher order terms in the expansion \eqref{cexpansion}. This started with  the pioneering work of Cimmino, who posed the question of classifying domains in $\mathbb R^3$ with property \eqref{densitycondition}. The question was first taken up by Nitsche in \cite{Ni89}, who observed that $\sigma^{(1)}=
\sigma^{(3)}=0$ for all domains, and that $\sigma^{(2)}$ is a constant multiple of the function $\Delta H+2H(H^2-K)$: but this expression vanishes for all minimal surfaces, which adds nothing new and leaves the question unsettled.  Finally in \cite{Ni95} Nitsche proved:

\begin{theorem}\label{Nitsche}

The only domains in $\mathbb R^3$ with the $\frac 12$-uniform density property have boundary congruent to the plane or the right minimal  helicoid.

\end{theorem}

This was achieved by proving, by a rather delicate argument, that $\sigma^{(4)}$ vanishes only when $\bd\Omega$ is a plane or the minimal helicoid, which finally answered Cimmino's question; taking into account the above equivalence, Theorem \ref{MPS} is then proved.

We will actually classify $\frac 12$-uniformly dense domains in all $3$-dimensional space forms (see Theorem \ref{Nitscheg}).

\subsection{The strategy of proof; the Dirichlet temperature and the divergence condition}\label{Subsec:Approach} Unlike the procedure we have seen in $\mathbb R^n$, which makes use of the equivalence between the $\frac 12$-property and the $\frac12$-uniform density, we directly work with the Cauchy temperature $u_C(t,x)$ of a domain in a general stochastically complete manifold $M$, by relating it with what we call the {\it Dirichlet temperature} function $u_D(t,x)$,  that is, the temperature at time $t>0$, at the point $x\in\Omega$ assuming that the initial temperature distribution is uniformly equal to $1$ on $\Omega$, and assuming Dirichlet boundary conditions: 
\begin{align}\label{Eq:DirichletTemperature}
        \begin{cases}
            \left(\frac{\partial}{\partial t} + \Delta\right) u_D=0 & \inn (0,+\infty)\times \Omega\\
            u_D=1 & \onn \{0\}\times\Omega\\
            u_D=0 & \onn (0,+\infty)\times \partial \Omega.
        \end{cases}
\end{align}

In fact we show that, if $\Omega$ has the $\frac 12$-property, then on $\Omega$ one has the identity:
\begin{equation}\label{ucud}
u_D=2u_C-1.
\end{equation}  
It is then proved that the {\it Dirichlet heat flow}, that is, the function on $(0,\infty)\times\bd\Omega$ given by
$\derive{u_D}{N}(t,x)$ (here $N$ is the inner unit normal vector to $\bd\Omega$) has an asymptotic expansion, as $t\to 0$, of type:
$$
\frac{\partial u_D}{\partial N} (t,x) \sim \frac{1}{\sqrt{\pi}} t^{-\frac{1}{2}}+ \sum_{k\geq 0} \gamma_k(x) t^{\frac{k}{2}} \quad as\ t\to 0
$$
for certain computable smooth invariants $\gamma_k\in C^{\infty}(\bd\Omega)$. Condition \eqref{ucud} is seen to force all even invariants to vanish:
$$
\gamma_{2n}=0 \quad\text{for all $n=0,1,2,\dots$}
$$
The condition $\gamma_0=0$ implies that $\bd\Omega$ is minimal (in any stochastically complete manifold) thus proving Theorem \ref{Thm_MinimalityBoundary}. When $\Omega$ is a domain in a $3$-dimensional space-form,  the condition $\gamma_4=0$ implies what we call the {\it divergence condition} on the boundary of the domain, namely the following geometric PDE
\begin{equation}\label{divcondition}
{\rm div}(S(\bar\nabla K))=0
\end{equation}
where $\bar\nabla$ is the gradient, $S$ is the shape operator and $K$ is the Gaussian curvature of $\bd\Omega$. 

\smallskip

Examination of the divergence condition on minimal surfaces of space-forms leads to the fact that $\bd\Omega$ is geodesically ruled and, by known results, $\bd\Omega$ is then totally geodesic or an embedded minimal helicoid. On the other hand, the symmetries of these particular surfaces imply that the two domains they bound have the $\frac 12$-property. 
Here is the conclusion, which implies in particular the main theorem, Theorem \ref{Thm_Main}.

\begin{theoremA}\label{divergence} Let $M^3_{\sigma}$ be a $3$-dimensional space-form and $\Omega\subseteq M_{\sigma}$ a smooth domain. The following are equivalent:

a) $\Omega$ has the $\frac 12$-property;

b) $\bd\Omega$ is minimal and satisfies the divergence condition \eqref{divcondition};

c) $\bd\Omega$ is totally geodesic or an embedded minimal helicoid, as in the list in Theorem \ref{Thm_Main}.
\end{theoremA}

Finally, the definition of $\frac 12$-uniform density extends to domains in space-forms. A posteriori, we remark the following fact, which reproves Nitsche result \cite{Ni95} and extends it to domains in space-forms. Its proof is postposed in Appendix \ref{Appendix:Nitsche}

\begin{theoremA}\label{Nitscheg} Let  $\Omega\subseteq M^3_{\sigma}$.  Then $\Omega$ is $\frac 12$-uniformly dense in $\bd\Omega$ if and only if its boundary is totally geodesic or an embedded minimal helicoid. 
\end{theoremA}

\subsection{Related results on other overdetermined heat equations}
Previous work involves the study of the Dirichlet temperature function $u_D(t,x)$ defined in \ref{Eq:DirichletTemperature}.

A hypersurface $\Gamma\subseteq\Omega$ is called {\it isothermic at time $t_0$} if
$$
u_D(t_0,x)=c \quad\text{for all $x\in\Gamma$}
$$
where $c$ is a constant (possibly depending on time $c=c(t_0)$). $\Gamma$ is said to be {\it stationary isothermic} if it is isothermic at all times $t_0>0$. For bounded domains in $\mathbb R^n$, it was conjectured (\textit{Matzoh ball soup} \cite{Kl64, Za87}) that if all isothermic hypersurfaces are stationary, then $\Omega$ must be a ball. The conjecture was proved by Alessandrini in \cite{Al90}; 
a remarkable improvement was achieved by Magnanini and Sakaguchi \cite{MS02} where the same conclusion is obtained only assuming the existence of one stationary isothermic hypersurface. 

\smallskip

Let $N$ be the inner unit normal to the (bounded) domain $\Omega$ of an arbitrary Riemannian manifold $M$. The normal derivative of the Dirichlet temperature
$\derive{u_D}{N}(t,x)$
at a point $x\in\bd\Omega$ is called the {\it Dirichlet heat flow} at $x$: it measures the rate at which $\Omega$ is loosing heat near $x$. If this function is spatially constant on $\bd\Omega$  at every fixed time $t$
$$
\derive{u_D}{N}(t,x)=\psi(t)\quad\text{for all $x\in\bd\Omega$}
$$
for a smooth function $\psi=\psi(t)$, then we say that $\Omega$ has the {\it constant flow property}. The structure of domains (in an analytic manifold $M$) having the constant flow property has been clarified by the second author in \cite{Sa16} and \cite{Sa18}. It turns out that these domains are {\it isoparametric tubes}, in the sense that they admit a singular Riemannian foliation whose regular leaves are parallel hypersurfaces of constant mean curvature, whose singular variety (in fact, the hot spot of $u_D$) is a smooth minimal submanifold at constant distance to $\bd\Omega$. In proving that result, a crucial role is played by the asymptotic expansion of the Dirichlet heat flow as $t\to 0$, developped in \cite{Sa04},  which will be generalized in this paper to the unbounded case and used in the proof.

\smallskip

Concerning the Cauchy temperature $u_C$, extensive work when $M=\mathbb R^3$ was done in \cite{MPSS16} and \cite{CSU23}. In particular, classification of unbounded domains admitting stationary isothermic surfaces was done in the paper \cite{MPSS16}, where the results rely on fine properties of minimal surfaces of $\mathbb R^3$. 
Finally, interesting generalizations of the above problems to two-phase conductors has been done in \cite{CSU23, Sak20, Sak24}.

\color{black}

\subsection{Scheme of the paper}\label{scheme} In Section \ref{mpt} we state two important general results on the heat equation on unbounded smooth domains of complete manifolds, which will be used in the main theorem. In Section \ref{Sec:SufficientConditions} we prove that domains bounded by an embedded minimal helicoid (or a totally geodesic surface) do have the $\frac 12$-property. In Section \ref{Sec:Rigidity} we relate the Cauchy and Dirichlet temperatures of $\frac 12$-domains, and prove that in a stochastically complete manifold $\frac 12$-domains have minimal boundary.  Section \ref{Sec:divergencecondition} consists in the main step of the proof: by using the asymptotic expansion of the Dirichlet heat flow, one gets that the boundary of a $\frac 12$-domain must satisfy the divergence condition \eqref{divcondition}. Section \ref{rigiditydc} is devoted to the classification of minimal surfaces having the divergence condition, which concludes the proof of the main theorem. The Appendix is devoted to the proof of the theorems of Section \ref{mpt}.

\section{Main preparatory theorems}\label{mpt}

There are two key ingredients for proving the main theorem, which are of independent interest. The first is the asymptotic expansion of the pointwise Dirichlet heat flow of (possibly unbounded) smooth domains of geodesically complete Riemannian manifolds.

\begin{theoremA}\label{Thm:UnboundedSavoExpansion}
Let $(M,\scal{\cdot}{\cdot})$ be a geodesically complete Riemannian manifold without boundary and $\Omega\subset M$ a (possibly unbounded) smooth domain. Denote by $u_D$ the Dirichlet temperature of $\Omega$ with initial data $u_0\in C^\infty(\Omega)$, i.e. the solution to
\begin{align}\label{Eq:DirichletTemperature}
        \begin{cases}
            \left(\frac{\partial}{\partial t} + \Delta\right) u_D=0 & in\ (0,+\infty)\times \Omega\\
            u_D=1 & on\ \{0\}\times\Omega\\
            u_D=0 & on\ (0,+\infty)\times \partial \Omega.
        \end{cases}
\end{align}
Then, the following asymptotic expansion holds for every $x\in \partial \Omega$
\begin{align}\label{Eq:UnboundedSavoExpansion}
\frac{\partial u_D}{\partial N} (t,x) \sim \frac{1}{\sqrt{\pi}} t^{-\frac{1}{2}}+ \sum_{k\geq 0} \gamma_k(x) t^{\frac{k}{2}} \quad as\ t\to 0
\end{align}
for certain smooth invariants $\gamma_k\in C^{\infty}(\bd\Omega)$, which can be computed by an explicit recursive formula.
\end{theoremA}

The theorem extends Theorem 2.1 in \cite{Sa04}, proved when $\Omega$ is bounded. The recursive formula computing the $\gamma_k$'s is given in Definition 2.5 in \cite{Sa04}. 
The  proof of Theorem \ref{Thm:UnboundedSavoExpansion} is postponed in Appendix \ref{Sec:AppendixC}. 

\smallskip

The second key ingredient for the proof of Theorem \ref{Thm_MinimalityBoundary} is a uniqueness result for the Dirichlet problem for the heat equation. This property will be necessary in order to compare the Dirichlet and Cauchy temperatures of a fixed $\frac{1}{2}$-domain. 

\begin{theoremA}\label{Thm_UniquenessDirichletAppendix}
    Let $(M,\scal{\cdot}{\cdot})$ be a stochastically complete Riemannian manifold and $\Omega\subseteq M$ a (possibly unbounded) smooth domain. Let $v\in C^0([0,+\infty)\times \Omega)$ be a bounded distributional solution to the following boundary value problem
    \begin{align}\label{Eq_UniquenessAppendix}
        \begin{cases}
            \left(\frac{\partial}{\partial t} + \Delta\right) v=0 & in\ (0,+\infty)\times \Omega\\
            v=0 & on\ \{0\}\times\Omega\\
            v=0 & on\ (0,+\infty)\times \partial \Omega.
        \end{cases}
    \end{align}
Then, $v$ is the constant null function.
\end{theoremA}

To the best of our knowledge, in the general case of stochastically complete Riemannian manifolds, Theorem \ref{Thm_UniquenessDirichletAppendix} is new in the literature. Its proof involves the invertibility property of the Laplace transform over the class of continuous and bounded functions and a characterization of the stochastic completeness in terms of the uniqueness of the solution to a certain elliptic equation, see Theorem \ref{Thm:SCUniqueness}.

If the manifold is assumed to be geodesically complete, a similar result is obtained in \cite{VM15}. Here, the authors provided a uniqueness result in unbounded domains under some $L^2$-growth conditions involving the existence of a suitable \textit{T\"acklind function} (see also \cite[Theorem 1.1]{OR78} for the Euclidean case).




\section{Sufficient conditions for the $\frac 12$-property}\label{Sec:SufficientConditions}

The scope here is to prove:

\begin{theoremA}\label{Thm_SuffCondition}
Let $M=\rr^3, \sss^3$ or $\hh^3$ and $\Omega\subset M$ be a smooth domain whose boundary $\Sigma=\partial \Omega$ is a totally geodesic hypersurface or an embedded minimal helicoid. Then, $\Omega$ is a $\frac{1}{2}$-domain.
\end{theoremA}


\subsection{General criteria}
We give  sufficient criteria for having the $\frac 12$-property based on the existence of a suitable families of isometries of $M$.  We then apply this criteria to our cases. 

\smallskip

In what follows, we assume that $\Omega$ is a smooth domain in a stochastically complete manifold of arbitrary dimension $n$ and we adopt the following notation
\begin{align}
    \Omega_+ \doteq \Omega, \quad \quad \Omega_-\doteq M\setminus \overline{\Omega}\quad \quad \andd \quad \quad \Sigma\doteq\partial \Omega.
\end{align}
The strategy of the proof is similar to the one adopted in \cite[Appendix A]{CSU23} and it is based on the existence of enough isometries of the ambient manifold $M$ that preserve or invert $\Omega_+$ and $\Omega_-$. We repeatedly use the fact that, by stochastic completeness, there is a unique (bounded) solution of the heat equation on $M$ with given (bounded) initial data (see Theorem \ref{Thm_StandardSCUniqueness} in the Appendix). 

We let $u_+$ (resp. $u_-$) be the solution of the heat equation on $M$ with initial data $\chi_{\Omega_+}$ (resp. $\chi_{\Omega_-}$). Since $\chi_{\Omega_+}+\chi_{\Omega_-}=1$ a.e. we see, by uniqueness,
\begin{equation}\label{sumone}
u_+(t,x)+u_-(t,x)=1
\end{equation}
for all $t>0$ and $x\in M$. The first criterium will be applied to domains in space-forms bounded by a totally geodesic hypersurface.

\begin{lemma} \label{reflection} Assume that $\Psi$ is an isometry of $M$ taking $\Omega_+$ to 
$\Omega_-$ and fixing $\Sigma$ (that is, $\Psi(x)=x$ for all $x\in\Sigma$). Then $\Omega$ has the $\frac 12$-property. 
\end{lemma}

\begin{proof} Since $\Psi$ is an isometry, the function $u_+\circ\Psi$ is still a solution of the heat equation and, since $\Psi$ takes $\Omega_+$ to $\Omega_-$, its initial data is $\chi_{\Omega_-}$: by uniqueness, $u_+\circ\Psi=u_-$. Pick a point $x\in\Sigma$, so that $\Psi(x)=x$. Then, by \eqref{sumone} and what we have said
$$
\begin{aligned}
1&=u_+(t,x)+u_-(t,x)\\
&=u_+(t,x)+u_+(t,\Psi(x))\\
&=2u_+(t,x)
\end{aligned}
$$
which shows the assertion. 
\end{proof}

The next criterium is more general, and will be applied when the boundary is a minimal helicoid. 

\begin{theorem}\label{Lem_12_RototranslationsReflection}
    Let $M$ be a stochastically complete Riemannian manifold and $\Omega\subset M$ a smooth domain. Suppose there exist a subgroup of isometries $\mathcal{T}\leq \isom M$ and an isometry $\Psi\in \isom M$ such that
 \begin{enumerate}
        \item Every $T\in\mathcal{T}$ preserves both $\Omega_+$ and $\Omega_-$, hence $T(\Sigma)=\Sigma$;
        \item $\Psi$ sends $\Omega_+$ to $\Omega_-$ and viceversa;
        \item\label{3_Lem_12_RototranslationsReflection} for every $x\in \Sigma$ there exists $T\in \mathcal{T}$ such that $\Psi(x)=T(x)$.
    \end{enumerate}
    Then, $\Omega$ is a $\frac{1}{2}$-domain. Note that (3) holds in particular if $\mathcal{T}$ is transitive on $\Sigma$.
\end{theorem}

\begin{remark}
Note that taking $\mathcal{T}=\{1\}$ in Theorem \ref{Lem_12_RototranslationsReflection} we recover Lemma \ref{reflection}.
\end{remark}

\begin{proof}  Let $T\in\mathcal {T}$. As $T$ is an isometry, the function 
$u_+\circ T$ is still a solution of the heat equation on $M$. Since $T$ preserves $\Omega_+$,  $u_+\circ T$ has initial data $\chi_{\Omega_+}$. By uniqueness:
$$
u_+\circ T=u_+
$$
Similarly, our assumptions in (2) imply: 
$$
u_+\circ\Psi=u_-.
$$
Fix $t>0$ and $x\in\Sigma$; by assumption (3) there exists $T\in\mathcal{T}$ such that $T(x)=\Psi(x)$. 
Hence:
$$
u_+(t,x)=u_+(t, T(x))=u_+(t,\Psi(x))=u_-(t,x),
$$
and, from \eqref{sumone}, conclude immediately that $u_+(t,x)=\frac 12$ for all $t>0$ and $x\in\Sigma$. 
\end{proof}

\subsection{Proof of Theorem \ref{Thm_SuffCondition}}\label{Subsec:DensitySigmaConstant}
Assume first that $\Sigma$ is totally geodesic, and take $\Psi$ to be the reflection across $\Sigma$. Then $\Psi$ satisfies the conditions in Lemma \ref{reflection} and the assertion follows. If $\Sigma$ is an embedded minimal helicoid in $M_{\sigma}$ we will apply Theorem \ref{Lem_12_RototranslationsReflection} in each case $\sigma=-1,0,1$, as explained in the next subsections. This will complete the proof. 

\subsubsection{Right helicoid in $\mathbb R^3$}\label{Subsubsec:EucBallSplit}
It is the surface $\Sigma$ defined as
\begin{align}\label{Eq:DefEuclideanHelicoid}
\Sigma\doteq\{(x,y,z)\in \rr^3\ :\ x\sin(z)=y\cos(z)\},
\end{align}
which splits $\rr^3\setminus \Sigma$ in the two following disjoint open sets
\begin{equation}\label{Eq:DefEuclideanDomains}
\begin{split}
&\Omega_+\doteq\{(x,y,z)\in \rr^3\ :\ x\sin(z)>y\cos(z)\}\\
&\Omega_-\doteq\{(x,y,z)\in \rr^3\ :\ x\sin(z)<y\cos(z)\}.
\end{split}
\end{equation}
A parametrization of $\Sigma$ is given by $X:\rr\times \rr \to \rr^3$ so that
\begin{align}
X:(u,v)\mapsto \left(\begin{array}{c}
v \cos(u)\\
v \sin(u)\\
u
\end{array} \right).
\end{align}
It's easy to see that for any $\beta\in \rr$ the roto-translation
\begin{align}
T_\beta(x,y,z)\doteq \left[\begin{array}{ccc}
\cos(\beta) & -\sin(\beta) & 0\\
\sin(\beta) & \cos(\beta) & 0\\
0 & 0 & 1
\end{array} \right] \left( \begin{array}{c}
x \\ y \\ z 
\end{array}\right) + \left(\begin{array}{c}
0 \\ 0 \\ \beta
\end{array} \right) \in \isom{\rr^3}
\end{align}
acts by isometry on the helicoid $\Sigma$. Moreover, every $T_\beta$ preserves $\Omega_+$ and $\Omega_-$ and acts on $\Sigma$ as
\begin{align}
    T_\beta (X(u,v))=X(u+\beta,v).
\end{align}
Now consider $\Psi \in \textnormal{Isom}(\rr^3)$ given by the rotation of $\pi$ around the $x$-axis, $\Psi(x,y,z)=(x,-y,-z)$. One can see that $\Psi(\Omega_\pm)=\Omega_\mp$ and for every $w=X(u,v)\in\Sigma$:
\begin{align}
    \Psi(w)=X(-u,v)=T_{-2u}(w).
\end{align}


\subsubsection{Hyperbolic helicoids in $\mathbb H^3$}\label{Subsubsec:HypBallSplit}
Fix $\alpha>0$. As already observed, in the hyperboloid model of the hyperbolic space the helicoids $\Sigma^\alpha$ are minimal embeddes surfaces which can be parametrized by $
X^\alpha:\rr^2 \to \hh^3$ (see \cite{dCD83, Mor82, Wa19})
\begin{align}
X^\alpha:(u,v) &\mapsto \left(\begin{array}{c} \cosh(\alpha u) \cosh(v)\\ \sinh(\alpha u) \cosh(v)\\ \cos(u) \sinh(v)\\ \sin(u) \sinh(v) \end{array} \right).
\end{align}
As in the Euclidean case, every helicoid $\Sigma^\alpha$ of $\hh^3$ splits the space in two disjoint connected open sets $\Omega_+$ and $\Omega_-$ with common boundary $\Sigma^{\alpha}$, since $\Sigma^{\alpha}$ is complete (hence closed) in $\mathbb H^3$ (see \cite[Theorem 1]{Mor82}). For $\beta\in\mathbb R$, consider the isometry
\begin{align}
T_\beta^\alpha\doteq\left[\begin{array}{cccc}
\cosh(\alpha \beta) & \sinh(\alpha \beta) & 0 & 0\\
\sinh(\alpha \beta) & \cosh(\alpha \beta) & 0 & 0\\
0 & 0 & \cos(\beta) & -\sin(\beta)\\
0 & 0 & \sin(\beta) & \cos(\beta)
\end{array} \right] \in \isom{\hh^3},
\end{align}
which preserves the orientation of the helicoid (hence also the domains $\Omega_\pm$) and the isometry $\Psi(w,x,y,z)=(w,-x,y,-z)$, which inverts $\Omega_+$ and $\Omega_-$. One verifies that for every $\beta\in \rr$ and for every $(u,v)\in \rr^2$
\begin{align}
    T_\beta^\alpha(X^\alpha(u,v))=X^\alpha(u+\beta,v) \quad \andd \quad \Psi(X^\alpha(u,v))=X^\alpha(-u,v).
\end{align}
Hence, for every $x=X^\alpha(u,v)\in \Sigma^\alpha$ one has
\begin{align}
\Psi(x)=X^\alpha(-u,v)=T^\alpha_{-2u}(x).
\end{align}

\subsubsection{Clifford torus in $\mathbb S^3$}\label{Subsubsec:SphBallSplit}
Let us consider the following parametrization of the Clifford torus $X:[0,2\pi)\times[0,2\pi) \to \sss^3$
\begin{align}
X:(u,v) & \mapsto \left(\begin{array}{c}
\cos(u) \cos(v)\\ \sin(u) \cos(v)\\ \cos(u) \sin(v) \\ \sin(u) \sin(v)\end{array} \right).
\end{align}
We denote by $\Omega_+$ and $\Omega_-$ the two connected components of $\sss^3\setminus \Sigma$. In particular, observing that
\begin{align}\label{Eq:DefSphericalHelicoid}
\Sigma\doteq \{(x_1,x_2,x_3,x_4)\in \sss^3\ :\ x_1 x_4=x_2 x_3\},
\end{align}
we have
\begin{equation}\label{Eq:DefSphericalDomains}
\begin{split}
&\Omega_+\doteq \{(x_1,x_2,x_3,x_4)\in \sss^3\ :\ x_1 x_4>x_2 x_3\}\\
&\Omega_-\doteq \{(x_1,x_2,x_3,x_4)\in \sss^3\ :\ x_1 x_4<x_2 x_3\}.
\end{split}
\end{equation}
We observe that $\mathcal{T}=\textnormal{SO}(2)\times \textnormal{SO}(2)< \isom{\sss^3}$ acts transitively on $\Sigma$ and preserves $\Omega_+$ and $\Omega_-$. Moreover, the map $\Psi(x_1,x_2,x_3,x_4)=(x_3,x_4,x_1,x_2)$, which is an isometry of $\sss^3$, inverts $\Omega_+$ and $\Omega_-$. Finally,  since the action of $\mathcal{T}$ on $\Sigma$ is transitive, for every $x\in \Sigma$ there exists $T\in \mathcal{T}$ such that $\Psi(x)=T(x)$.
\begin{remark}\label{Rmk_SphericalHelicoids}
We observe that the Clifford torus can be seen as an element of the family of the \textit{spherical helicoids}, that is, a one-parameter family of immersed hypersurfaces of  $\Sigma^\alpha$. For more details see Remark \ref{Rmk_CliffordTorus}.
\end{remark}


\section{Rigidity of the $\frac 12$-property: the heat flow invariants}\label{Sec:Rigidity}

In this section we consider the general case of a smooth domain $\Omega$ in a stochastically complete manifold $M$, and prove Theorem \ref{Thm_MinimalityBoundary}.
We use in a crucial way the asymptotic expansion of the Dirichlet heat flow, as defined in \eqref{Eq:UnboundedSavoExpansion}, and which we recall here.


\subsection{Description of the heat flow invariants} Let $u_D=u_D(t,x)$ be the Dirichlet temperature function as defined in \eqref{Eq:DirichletTemperature}. The heat flow of the domain $\Omega$ is the function on $(0,\infty)\times\bd\Omega$ given by the normal derivative
$
\derive{u_D}{N}(t,x).
$
Recall (see \eqref{Eq:UnboundedSavoExpansion}) that it admits an asymptotic series, as $t\to 0$:
$$
\derive{u_D}{N}(t,x)\sim\dfrac{1}{\sqrt{4\pi t}}+\sum_{k=0}^{\infty}\gamma_k(x)t^{k/2}
$$
for certain invariants $\gamma_k\in C^{\infty}(\bd\Omega)$ described as follows. Fix a small neighborhood $U\subset\Omega$ of a given point of the boundary $\Sigma$ where the distance function to the boundary:
$$
\rho(x)\doteq d(x,\bd\Omega)
$$
is smooth. Then, there exists a sequence of differential operators $\{D_n\}_{n\in \mathbb N}$ acting on $C^{\infty}(U)$ such that
\eqref{Eq:UnboundedSavoExpansion} can be written as
\begin{align}\label{Eq_Gamma_k}
\gamma_k(x)=\left( 1+\frac{k}{2}\right) D_{k+2}1(x)
\end{align}
for all $x\in\bd\Omega$: in other words, $\gamma_k$ is obtained by applying the operator $(1+\frac k2)D_{k+2}$ to the constant function $1$ and then restricting this function to $\bd\Omega$. 

\smallskip

To define these operators, consider the first order operator $\mathcal{N}$ which acts on the function $\phi\in C^{\infty}(U)$ as follows:
$$
\mathcal{N}\phi=2\scal{\nabla\phi}{\nabla\rho}-\phi\Delta\rho.
$$
Then, the operator $D_k$ is a homogenous polynomial of degree $k-1$ in the operator $\mathcal{N}$ (of degree one) and the Laplacian $\Delta$ of $\Omega$ (of degree two). 
The sequence $\{D_k\}$ is explicitly computable by a recursive formula (see Definition 2.5 in \cite{Sa04}). 
We list below the first few operators of even degree, which will be used later in this section (see \cite[Table 2.6]{Sa04})
\begin{align}\label{Eq_D2D4D6}
D_2 &= \frac{1}{2}\mathcal{N}  \nonumber \\
D_4 &=-\frac{1}{16}\left(\Delta \mathcal{N} + 3 \mathcal{N} \Delta\right)\\
D_6 &= \frac{1}{768} \left(\Delta \mathcal{N}^3 - \mathcal{N}^3\Delta + \mathcal{N} \Delta \mathcal{N}^2 - \mathcal{N}^2 \Delta \mathcal{N}\right.\\
& \quad \quad \quad \quad \left. + 40 \mathcal{N} \Delta^2 + 8 \Delta^2 \mathcal{N} + 16 \Delta \mathcal{N} \Delta \right). \nonumber
\end{align}
In what follows, we recall the shape operator of $\Sigma=\bd\Omega$, which is the symmetric endomorphism of $T\Sigma$ defined by the formula:
$$
S(X)=-\nabla_XN 
$$
for all $X\in T\Sigma$. Its trace $\eta={\rm tr}(S)$ will be called the (unnormalized) {\it mean curvature} of $\Sigma$. Note that the mean curvature depends on the choice of the unit normal vector. Finally, it is a classical fact that, on $\Sigma$, one has $\nabla\rho=N$ and $\Delta\rho=\eta$. In particular, since $D_21=\frac 12\mathcal{N}1=-\frac 12\Delta\rho$, we have
\begin{equation}\label{gammazero}
\gamma_0=-\frac 12\eta
\end{equation}
and $\gamma_0$ vanishes if and only if $\Sigma$ is minimal. In general, if $x\in U$, consider the hypersurface $\Sigma_x$ parallel to $\Sigma$ and passing through $x$, that is, 
$$
\Sigma_x=\rho^{-1}(\rho(x)).
$$
Then $\nabla\rho(x)$ is the inner unit normal to $\Sigma_x$ and $\Delta\rho(x)$ is the mean curvature of $\Sigma_x$ at $x$.


\subsection{Relation between $u_C$ and $u_D$ when $\Omega$ has the $\frac 12$-property} Now assume that $\Omega$ has the $\frac 12$-property, so that $u_C(t,x)=\frac 12$ for all $t>0$ and $x\in\bd\Omega$. Consider the functions on $\Omega$:
$$
u_D(t,x), \quad w(t,x)\doteq 2u_C(t,x)-1.
$$
They are both solutions of the heat equation on $\Omega$, they have the same initial data
(constant, equal to $1$), and both have Dirichlet boundary data. Hence the function
$$
v(t,x)=u_D(t,x)-w(t,x)
$$
is a solution to the Dirichlet boundary value problem on $\Omega$ with zero initial data. By the regularity of $v$ up to $t=0$ (see \cite[Theorem 7.16]{Gr09}), we can apply the uniqueness result of Theorem \ref{Thm_UniquenessDirichletAppendix} obtaining that that $v(t,x)\equiv 0$. That is, one has:
$$
u_D(t,x)=2u_C(t,x)-1
$$
for all $t>0$ and $x\in\Omega$. Recall the splitting:
$$
\Omega_+\doteq\Omega, \quad \Omega_-\doteq M\setminus\bar\Omega, \quad \Sigma\doteq \bd\Omega
$$
and let $u_D^+$ (resp. $u_D^-$) denote the Dirichlet temperature of $\Omega_+$ (resp. $\Omega_-$). Let $N_+$ (resp. $N_-$) be the respective inner unit normal vectors; clearly $N_-=-N_+$. Therefore, denoting $\eta_+$ (resp. $\eta_-$) the mean curvature of $\Sigma$ as boundary of $\Omega_+$ (resp. $\Omega_-$) we have clearly $\eta_-=-\eta_+$ on $\Sigma$.

Then we have the following 

\begin{lemma} Assume that $\Omega$ has the $\frac 12$-property and let $u_D^+$ and $u_D^-$ be as above. Then the heat flow functions of $\Omega_+$ and $\Omega_-$ coincide on $\Sigma=\bd\Omega_+=\bd\Omega_-$:
$$
\derive{u_D^+}{N_+}(t,x)=\derive{u_D^-}{N_-}(t,x)
$$
for all $t>0$ and $x\in\Sigma$. In particular, the respective heat flow invariants coincide:
$$
\gamma_k^+(x)=\gamma_k^-(x)
$$
for all $x\in\Sigma$ and $k\geq 0$.
\end{lemma}

\begin{proof} We have already observed that on $\Omega^+$ we have:
$$
u_D^+=2u_C-1.
$$
Observe that, $u_C(0,x)=0$ for all $x\in\Omega_-$. Then, same argument shows that on $\Omega_-$ we have:
$$
u_D^-=1-2u_C.
$$
Therefore we conclude that, for all $t>0$ and $x\in\Sigma$:
$$
\derive{u_D^+}{N_+}(t,x)=2\derive{u_C}{N_+}(t,x)=-2\derive{u_C}{N_-}(t,x)=\derive{u_D^-}{N_-}(t,x),
$$
the second equality following because $u_C(t,\cdot)$ is smooth on $M$ and $N_-=-N_+$.
\end{proof}


\subsection{Proof of Theorem \ref{Thm_MinimalityBoundary}}

We can now prove:
\begin{theorem}
Let $\Omega$ be a (possibly unbounded) domain in a stochastically and geodesically complete $n$-dimensional Riemannian manifold $(M,\scal{\cdot}{\cdot})$. If $\Omega$ has the $\frac{1}{2}$-property, then $\partial\Omega$ is minimal. 
\end{theorem}

\begin{proof} Recall from the previous lemma that $\gamma_0^+=\gamma_0^-$. On the other hand we observed in \eqref{gammazero}
$$
\gamma_0^+=-\frac 12\eta_+=\frac 12\eta_-=-\gamma_0^-.
$$
Hence $\gamma_0^+=\gamma_0^-=0$, that is, $\eta_+=\eta_-=0$ and the boundary is minimal. 
\end{proof}


\section{The divergence condition}\label{Sec:divergencecondition}

In the present section we always assume $\Omega$ to be a smooth domain in a $3$-dimensional space form $M^3_{\sigma}$ of constant curvature $\sigma\in \{-1,0,1\}$. 

\smallskip

We have the following calculation. Here $\gamma_k$ is the $k$-th coefficient in the asymptotic expansion \eqref{Eq:UnboundedSavoExpansion}. 

\begin{theorem}\label{Thm_Gamma4_Divergence}  Assume that $\Sigma=\bd\Omega$ is minimal. Then $\gamma_0=\gamma_2=0$ and 
$$
\gamma_4=\dfrac{5}{16}{\rm div}(S(\bar\nabla K))
$$
where ${\rm div}$ and $\bar\nabla$ denote the divergence  and gradient in $\Sigma$, and $K$ the gaussian curvature of $\Sigma$.
\end{theorem}

The proof is given below.

\begin{cor} Assume  that $\Omega$ has the $\frac 12$-property. Then $\Sigma$ is minimal and 
$$
{\rm div}(S(\bar\nabla K))=0
$$
holds everywhere on $\Sigma$.
\end{cor}

\begin{proof} The minimality has been established in the previous section. In the notation of the previous section $\Omega_+\doteq\Omega, \Omega_- \doteq M\setminus\bar\Omega_+, \Sigma \doteq \bd\Omega_+=\bd\Omega_-$ we know that $\gamma_4^+=\gamma_4^-$.  Let $S^+$ (resp. $S^-$) be the shape operator of $\Sigma$ as boundary of $\Omega^+$ (resp. $\Omega_-$). Hence $S^-=-S^+$; as $K$ does not depend on the choice of the normal, we have:
$$
\gamma_4^+=\dfrac{5}{16}{\rm div}(S^+(\bar\nabla K))=-\dfrac{5}{16}{\rm div}(S^-(\bar\nabla K))=-\gamma^-_4.
$$
This gives $\gamma_4^+=\gamma_4^-=0$ as asserted. 
\end{proof}


\subsection{Preparatory results for Theorem \ref{Thm_Gamma4_Divergence}}

As a preparatory step for the divergence condition, we introduce some formalism and prove some commutation formulae. 

We will consider the foliation of a small tubular neighborhood $U$ of $\partial\Omega$ by the level sets of the distance function $\rho=d(\cdot,\partial\Omega)$. The vector field $N\doteq\nabla\rho$ is smooth on $U$, has unit length, and is everywhere normal to the leaves. A vector field $X$ on $U$ is {\it tangential} if $\scal{X}{N}=0$ on $U$, so that $X$ is everywhere tangent to the leaves. One can always split a vector field $X$ as 
$$
X=\bar X+\scal{X}{N}N
$$
with $\bar X$ tangential. In particular, we denote $\bar\nabla\phi$ the tangential gradient of $\phi$.
The shape operator of $\partial\Omega$ extends to a field of endomorphisms, acting on tangential fields,  by setting $S(\bar X)=-\nabla_{\bar X}N$: at $p\in\rho^{-1}(r)$ it is clearly the shape operator of the parallel hypersurface $\rho^{-1}(r)$ and its trace:
$$
\eta\doteq \tr S
$$ is the mean curvature of $\rho^{-1}(r)$. 
Let $\bar\delta$ denote the tangential divergence: if $(e_1, e_2)$ is a local orthonormal frame for the leaf $\rho^{-1}(r)$ through a given point $p\in U$, one has, at $p$:
$$
\bar\delta X=-\sum_{j=1}^n\scal{\bar\nabla_{e_j}\bar X}{e_j}.
$$
One has the usual formula
for the tangential Laplacian $\bar\Delta\phi=\bar\delta\bar\nabla\phi$. According to this definition, one has for $\phi\in C^{\infty}(U)$ the splitting
$$
\Delta\phi=\bar\Delta\phi+\Delta_{R}\phi,
$$
where:
$$
\Delta_{R}\phi=-\deriven{2}{\phi}{N}+\eta\derive{\phi}.{N}
$$
For the following commutation rules, we refer to page 431 of \cite{Sa01}. For all tangential vector fields $\bar X$ on $U$:
\begin{equation}\label{bardelta}
\derive{}{N}\bar\delta\bar X=\bar\delta(\nabla_N\bar X)+\scal{\bar 
X}{\bar\nabla\eta}+\bar\delta(S(\bar X)).
\end{equation}
Moreover for all $\phi\in C^{\infty}(U)$ one has, on $U$:
\begin{align}\label{bardeltatwo}
\twosystem{\nabla_N\bar\nabla\phi=\bar\nabla(\derive{\phi}{N})+S(\bar\nabla\phi)}
{\derive{}{N}\bar\Delta\phi=\bar\Delta\Big(\derive{\phi}{N}\Big)+\scal{\bar\nabla\phi}{\bar\nabla\eta}+2\bar\delta(S(\bar\nabla\phi)).}
\end{align}

Having said that, we prove the following

\begin{lemma}\label{Thm_2NormalDerivativeLaplacian}
Let $\Omega$ be a smooth domain in a  Riemannian manifold and let $\phi\in C^{\infty}(U)$, where $U$ is a neighborhood of $\partial\Omega$ as above. If both $\phi$ and $\eta$ are constant on $\partial\Omega$, then one has, on $\partial\Omega$:
\begin{align}
\frac{\partial^2}{\partial N^2} \bar\Delta\phi=\bar\Delta \left(\frac{\partial^2 \phi}{\partial N^2} \right)+4 \bar\delta \left(S\left(\bar\nabla \left(\frac{\partial \phi}{\partial N} \right) \right) \right).
\end{align}
\end{lemma}
\begin{proof}
We start by observing that by \eqref{bardeltatwo} we have on $\partial\Omega$:
\begin{align}\label{first}
\deriven{2}{}{N}\bar\Delta\phi=\derive{}{N}\bar\Delta(\derive{\phi}{N})+\derive{}{N}\scal{\bar\nabla\phi}{\bar\nabla\eta}+2\derive{}{N}\bar\delta(S(\bar\nabla\phi))
\end{align}
Now by \eqref{bardeltatwo} and the fact that $\eta$ is constant on $\partial\Omega$:
\begin{align}\label{second}
\derive{}{N}\bar\Delta(\derive{\phi}{N})&=\bar\Delta \left(\frac{\partial^2 \phi}{\partial N^2} \right)+\scal{\bar\nabla\derive{\phi}{N}}{\bar\nabla\eta}+2\bar\delta\left(S(\bar\nabla(\derive{\phi}{N})\right)\\
&=\bar\Delta \left(\frac{\partial^2 \phi}{\partial N^2} \right)+
2\bar\delta\left(S(\bar\nabla(\derive{\phi}{N})\right)\\
\end{align}
Since $\phi$ and $\eta$ are both constant on $\partial\Omega$:
\begin{equation}\label{third}
\derive{}{N}\scal{\bar\nabla\phi}{\bar\nabla\eta}=0
\end{equation}
and moreover, by \eqref{bardelta} applied to $\bar X=S(\bar\nabla\phi)$:
\begin{align}\label{fourth}
2\derive{}{N}\bar\delta(S(\bar\nabla\phi))=&
2\bar\delta\left(\nabla_N(S(\bar\nabla\phi))\right)+2\scal{S(\bar\nabla\phi)}{\bar\nabla\eta}\\
&+2\bar\delta\left(S^2(\bar\nabla\phi)\right)\\
=&2\bar\delta\left(\nabla_N(S(\bar\nabla\phi))\right)\\
=&2\bar\delta\left(S(\bar\nabla(\derive{\phi}{N})\right)
\end{align}
We need to justify the very last passage. By definition of covariant derivative, we have, on $\partial\Omega$
\begin{align}\label{fifth}
\nabla_N(S(\bar\nabla\phi))&=\nabla_NS(\bar\nabla\phi)+S\left(\nabla_N(\bar\nabla\phi)\right)\\
&=S^2(\bar\nabla\phi))+R_N(\bar\nabla\phi)+S\left(\bar\nabla(\derive{\phi}{N})+S(\bar\nabla\phi)\right)\\
&=S\left(\bar\nabla(\derive{\phi}{N})\right)
\end{align}
and then
$$
2\bar\delta\left(\nabla_N(S(\bar\nabla\phi))\right)=2\bar\delta\left(S\left(\bar\nabla(\derive{\phi}{N})\right)\right)
$$
which proves \eqref{fourth}. 

We now add \eqref{second}, \eqref{third} and \eqref{fourth} and substitute in \eqref{first}. The proof is complete.
\end{proof}

We now specialize to the $3$-dimensional situation. We first observe: 

\begin{lemma}\label{Lem_EvenNormalDerivativesEta}
Let $\Omega$ be a smooth domain with minimal boundary in $M_{\sigma}$. Then, all even normal derivatives of $\eta$ vanish on $\partial \Omega$.
\end{lemma}
\begin{proof}  A well-known fact gives $\nabla_NS=S^2+R_N$, where $R_N(X)\doteq R(N,X)N$ (here $R$ is the Riemann tensor of $M^3_{\sigma})$. Since $R_N(X)=\sigma X$, we have
\begin{equation}\label{normalshape}
\nabla_NS=S^2+\sigma I.
\end{equation}
Recall that by definition $\eta=\tr{S}$; as $\derive{}{N}{\rm tr}=\tr{\nabla_N}$ we see
$$
\derive{\eta}{N}=\tr{\nabla_N S}=\textnormal{tr}(S^2+\sigma)=\textnormal{tr}(S^2)+2
\sigma.
$$
Successive application of \eqref{normalshape} gives:
$$
\dfrac{\partial^2\eta}{\partial N^2}=2{\rm tr}(S^3)+2\sigma{\rm tr}(S).
$$
Writing for simplicity $\eta^{(k)}\doteq\dfrac{\partial^k\eta}{\partial N^k}$, standard induction shows that there are real coefficients $a_j,b_j$ such that, on $\partial\Omega$:
\begin{align}\label{Eq_Derivatives_eta}
    \eta^{(2k)}=\sum_{j=0}^{k} a_j {\rm tr}(S^{2j+1}) \quad \andd \quad \eta^{(2k+1)}=\sum_{j=0}^{k+1} b_j {\rm tr}(S^{2j}).
\end{align}
If $k_1$ and $k_2$ are the principal curvatures, we have 
${\rm tr}(S^n)=k_1^n+k_2^n$, and then ${\rm tr}(S^n)=0$ for all $n$ odd, because $k_2=-k_1$. Then $\eta^{(2k)}=0$ for all $k$, as asserted. 
\end{proof}

As a consequence, we have

\begin{theorem}\label{Thm_TechnicalDivergenceConditionMinimalHypersurfaces}
Let $\Omega$ be a domain with minimal boundary in $M^3_{\sigma}$. Then, at every point of $\partial\Omega$ we have:
\begin{align}
\frac{\partial^2}{\partial N^2} \bar\Delta\eta=-8\bar\delta \left( S\left( \bar\nabla K\right) \right),
\end{align}
where $K$ is the Gaussian curvature of $\Sigma=\partial \Omega$.
\end{theorem}
\begin{proof} Denote by $k_1$ and $k_2$ the principal curvatures of $\Sigma$, so that $k_1=-k_2\doteq k$. Since ${\rm{tr}}(S^2)=2k^2$,
by what we have seen in the proof of Lemma \ref{Lem_EvenNormalDerivativesEta}
\begin{align}
\frac{\partial\eta}{\partial N}=\textnormal{tr}(S^2)+2\sigma=2k^2+2\sigma\quad \quad \textnormal{and} \quad \quad \frac{\partial^2 \eta}{\partial N^2}=0.
\end{align}
By the Gauss' Lemma, $K=\sigma-k^2$ and therefore
\begin{align}
\bar\nabla\left( \frac{\partial \eta}{\partial N}\right)=2\bar\nabla(k^2)=-2\bar\nabla K.
\end{align}
Applying Theorem \ref{Thm_2NormalDerivativeLaplacian} to the function $\phi=\eta$, we obtain
\begin{align}
\frac{\partial^2}{\partial N^2} \bar\Delta\eta & =\bar\Delta \left(\frac{\partial^2 \eta}{\partial N^2}\right)+ 4\bar\delta\Bigg( S \left( \bar\nabla \left( \frac{\partial \eta}{\partial N}\right) \right) \Bigg)=-8 \bar\delta\left( S\left( \bar\nabla K\right) \right).
\end{align}
\end{proof}


\subsection{Proof of Theorem \ref{Thm_Gamma4_Divergence}} Recall that $\gamma_4(x)=3D_61(x)$ for all $x\in\Sigma$, where (recall \eqref{Eq_D2D4D6}) 
\begin{align}\label{Eq_AppD6}
D_61= \frac{1}{768} \left(-\Delta \mathcal{N}^2 \eta - \mathcal{N} \Delta \mathcal{N} \eta + \mathcal{N}^2 \Delta \eta - 8 \Delta^2 \eta \right).
\end{align}
The proof of Theorem \ref{Thm_Gamma4_Divergence} will follow from the identity
\begin{equation}\label{dsixone}
D_6 1=-\frac{80}{768} \bar\delta\left( S\left( \bar\nabla K\right) \right),
\end{equation}
where $K$ is the Gaussian curvature of $\Sigma=\partial\Omega$.

\smallskip

In what follows we set for simplicity of notation:
$$
\deriven{k}{\phi}{N}\doteq \phi^{(k)}
$$
and recall the splitting 
\begin{align}
\Delta=\Delta_R + \bar\Delta,
\end{align}
where the operator $\Delta_R$ acts on $C^{\infty}(U)$ by:
\begin{align}
\Delta_R \phi=-\frac{\partial^2 \phi}{\partial N^2}+\eta \frac{\partial \phi}{\partial N}= -\phi'' + \eta \phi'.
\end{align}
We proceed by examining each term of \eqref{Eq_AppD6}:

{\bf First term: $\Delta \mathcal{N}^2 \eta$.}
\begin{align}
\mathcal{N}^2 \phi = 4\phi'' - 4\eta \phi' - 2\eta' \phi + \eta^2 \phi.
\end{align}
Recall in what follows that, by Lemma \ref{Lem_EvenNormalDerivativesEta}, all even normal derivatives of $\eta$ vanish on $\Sigma$. Hence, on $\Sigma$ we have $\mathcal{N}^2\eta = 0$ and so
\begin{align}
\bar\Delta\mathcal{N}^2 \eta =0.
\end{align}
Moreover, on $U$:
\begin{align}
\Delta_R \mathcal{N}^2 \eta &=-(\mathcal{N}^2 \eta)''\\
&=-4\eta^{(4)}+6 \eta \eta^{(3)} + 18 \eta' \eta'' - 6 \eta (\eta')^2 - 3 \eta^2 \eta''
\end{align}
so that, $\Delta_R \mathcal{N}^2 \eta=0$ on $\Sigma$, implying
\begin{align}
\Delta \mathcal{N}^2 \eta =0.
\end{align}


{\bf Second term: $\mathcal{N} \Delta \mathcal{N} \eta$.}
\begin{align}
\phi'=2\eta''-2\eta\eta',
\end{align}
it follows that $\phi'=0$ on $\Sigma$ and hence $\bar\Delta \phi'=0$. Using that $\Sigma$ is minimal, by \eqref{bardeltatwo} we get
\begin{equation}\label{Eq_App2_Thm_MainDivergenceConditionMinimalHypersurfaces}
\begin{split}
\mathcal{N}\bar\Delta \phi &=2 \frac{\partial}{\partial N} \bar\Delta \phi\\
&=2 \left( \bar\Delta \phi' + 2 \bar\delta (S(\bar\nabla\phi))\right)\\
&=4 \bar\delta(S(\bar\nabla \phi)).
\end{split}
\end{equation}
On $\Sigma$ one has $\phi=2\eta'-\eta=2\eta'$ and hence $\bar\nabla \phi=2 \bar\nabla \eta'$. As seen in the proof of Theorem \ref{Thm_TechnicalDivergenceConditionMinimalHypersurfaces} we have $\bar\nabla \eta'=-2 \bar\nabla K$ and so $\bar\nabla \phi=-4 \bar\nabla K$, that implies
\begin{align}\label{Eq_App3_Thm_MainDivergenceConditionMinimalHypersurfaces}
\mathcal{N}\bar\Delta \mathcal{N} \eta=-16 \delta(S(\bar\nabla K))
\end{align}
thanks to \eqref{Eq_App2_Thm_MainDivergenceConditionMinimalHypersurfaces}. Lastly, observe that on $\Sigma$
\begin{align}
\mathcal{N}\Delta_R \mathcal{N}\eta &=2(\Delta_R \mathcal{N}\eta)'\\
&=-4 \eta^{(4)} + 16 \eta' \eta'' + 8 \eta \eta^{(3)} - 8 \eta (\eta')^2 - 4 \eta^2 \eta'',
\end{align}
so that, $\mathcal{N}\Delta_R \mathcal{N}\eta =0$ on $\Sigma$ and hence, together with \eqref{Eq_App3_Thm_MainDivergenceConditionMinimalHypersurfaces},
\begin{align}
\mathcal{N} \Delta \mathcal{N} \eta &=\mathcal{N} \bar\Delta \mathcal{N} \eta+\mathcal{N} \Delta_R \mathcal{N} \eta = - 16 \bar\delta(S(\bar\nabla K)).
\end{align}


{\bf Third term: $\mathcal{N}^2 \Delta \eta$}
\begin{align}
\mathcal{N}^2 \Delta_R \eta
&=-4\eta^{(4)} + 14 \eta' \eta'' + 8 \eta \eta^{(3)} - 6 \eta (\eta')^2 - 5 \eta^2 \eta'' + \eta^3 \eta'
\end{align}
which implies $\mathcal{N}^2 \Delta_R \eta=0$ on $\Sigma$. If we set $\phi=\bar\Delta \eta$, then on $\Sigma$ we have $\phi=0$ and so
\begin{align}
\mathcal{N}^2 \phi= 4 \phi'' = -32 \bar\delta(S(\bar\nabla K))
\end{align}
by Theorem \ref{Thm_TechnicalDivergenceConditionMinimalHypersurfaces}. This, together with previous calculation, implies
\begin{align}
\mathcal{N}^2 \Delta \eta &=\mathcal{N}^2 \bar\Delta \eta =-32 \bar\delta(S(\bar\nabla K)).
\end{align}


{\bf Fourth term: $\Delta^2 \eta$}
\begin{align}
\Delta^2 \eta=\Delta_R \bar\Delta \eta+(\bar\Delta)^2 \eta + \Delta_R^2 \eta +\bar\Delta \Delta_R \eta.
\end{align}
Clearly $(\bar\Delta)^2\eta=0$. Now $\Delta_R\eta=-\eta''+\eta\eta'=0$ on $\Sigma$ and hence the fourth term vanishes. Since
\begin{align}
\Delta_R^2 \eta 
&=\eta^{(4)} -2 \eta \eta^{(3)} - 3 \eta' \eta''+ \eta^2 \eta'' + \eta (\eta')^2=0
\end{align}
on $\Sigma$, we are left with
\begin{align}
\Delta^2\eta &=\Delta_R \bar\Delta \eta=-(\bar\Delta \eta)''=8\bar\delta(S(\bar\nabla K)).
\end{align}
by Theorem \ref{Thm_TechnicalDivergenceConditionMinimalHypersurfaces}.

\smallskip 
Collecting all the contributions, the claim \eqref{dsixone} follows.

\section{Rigidity of the divergence condition}\label{rigiditydc}

The scope of this section is to prove the following result.

\begin{theorem}\label{dc} Let $\Sigma$ be minimal surface in $M_{\sigma}$ for which
$$
{\rm div}(S(\bar\nabla K))=0.
$$
Then $\Sigma$ is geodesically ruled, that is, it is foliated by geodesics of the ambient manifold. 
\end{theorem}

\noindent For the proof we use an isothermal parametrization of a special type.

\subsection{Isothermal parametrization with asymptotic coordinate lines} It is well-known that the umbilical points of a minimal surface are either isolated or make up the whole space $\Sigma$. In the second case $\Sigma$ is totally geodesic and we are done. In the first case, we fix any point and  consider an open neighborhood $A$ which is free of umbilical points and admits an isothermal parametrization with coordinates $(u,v)$. 

Following Nietsche \cite{Ni89} we can make a conformal change of coordinates so that the coordinates lines $u=c$ and $v=c$ are asymptotic lines, that is, if $\gamma=\gamma(t)$ is a parametrization of one such line, we have, if $l$ is the second fundamental form,
$$
l(\gamma'(t),\gamma'(t))=0
$$
for all $t$. This change of coordinates is found as follows: if the matrix of the second fundamental form is $\twomatrix LMMN$, then the Codazzi-Mainardi identity implies that 
the map $\Phi=L-iM$ is holomorphic; then, the conformal change is suitably constructed from $\Phi$ (see section 172 of \cite{Ni89}). This is valid for minimal surfaces of $\mathbb R^3$, but continues to hold in all space-forms because the Codazzi-Mainardi identity extends without change in those ambient spaces. 

The conclusion is that in a neighborhood of a non-umbilical point we have coordinates $(u,v)$ so that the matrices of the first and second fundamental form are, respectively,
\begin{equation}\label{isoc}
g=E\twomatrix 1001 \quad \textnormal{and}\quad l=\twomatrix 0110,
\end{equation}
where $E=E(u,v)$ is the conformal factor. Then the principal curvatures are $\frac 1E$ and $-\frac 1E$, and the Gauss lemma says that
\begin{equation}\label{gausslemma}
K=\sigma-\dfrac{1}{E^2}.
\end{equation}

\begin{lemma}\label{Lem_Identity_Derivatives_E} Assume that $\Sigma$ satisfies the divergence condition as in Theorem \ref{dc}.  In isothermal coordinates $(u,v)$ as in \eqref{isoc} one has the identity: 
$$
EE_{uv}-4E_uE_v=0.
$$
Hence, if $Q=E^{-3}$ we have $Q_{uv}=0$
identically on $\Omega$.
\end{lemma}

\begin{proof} We compute ${\rm div}(S(\bar\nabla K))$ in local coordinates. Now, setting $\bd_1\doteq\bd_u$ and $\bd_2\doteq \bd_v$ we have:
$$
\bar\nabla K=\dfrac{1}{E}\Big(K_u\bd_1+K_v\bd_2\Big)
$$
Now $S(\bd_1)=\frac 1E\bd_2$ and $S(\bd_2)=\frac 1E\bd_1$ by \eqref{isoc},  hence
$$
S(\bar\nabla K)=\dfrac{1}{E^2}\Big(K_v\bd_1+K_u\bd_2\Big)
$$
so that
$$
{\rm div}(S(\bar\nabla K))=\dfrac{1}{E}\Big(\bd_1(K_vE^{-1})+\bd_2(K_uE^{-1})\Big)
$$
From \eqref{gausslemma} we see $K_u=2E^{-3}E_u$ and $K_v=2E^{-3}E_v$
$$
\begin{aligned}
\bd_1(K_vE^{-1})&=2\bd_1(E^{-4}E_v)\\
&=-8E^{-5}E_uE_v+2E^{-4}E_{uv}.
\end{aligned}
$$
and similarly for the second term. We end-up with:
$$
{\rm div}(S(\bar\nabla K))=\dfrac{4}{E^6}\Big(EE_{uv}-4E_uE_v\Big)
$$
and the first identity follows. The second is a straightforward verification. 
\end{proof}
In any isothermal parametrization with conformal factor $E$ the Gauss curvature is expressed as
$$
K=\dfrac 1{2E}\Delta\log E
$$
(here $\Delta f=-f_{uu}-f_{vv}$). Using \eqref{gausslemma} we see
$$
\Delta\log E=2\sigma E-\frac{2}{E}.
$$
Set $Q=E^{-3}$ so that $E=Q^{-1/3}$ (we know that $Q_{uv}=0$). The PDE becomes:
$$
\Delta\log Q=-6\sigma Q^{-1/3}+6Q^{1/3}.
$$
It will be useful to normalize $Q$ so that it takes the value $1$ at $(0,0)$; we end-up with the following consequence of Lemma \ref{Lem_Identity_Derivatives_E}.

\begin{lemma}\label{sessantatre} Let $P=\frac{1}{Q(0,0)}Q$ and $k=Q(0,0)^{\frac 13}>0$. Then $P$ is analytic, positive on $A$ and satisfies
$$
\Delta\log P=\phi(P), \quad P(0,0)=1, \quad P_{uv}=0\quad\text{on}\quad A,
$$
where we set $\phi(x)=6kx^{1/3}-\dfrac{6\sigma}{k}x^{
-1/3}$.
\end{lemma}

The following fact will be proved in the appendix; its proof is inspired by a similar technical lemma in Nitsche (p. 242 in \cite{Ni95}). 

\begin{theorem}\label{Thm_oneofthevariables} An analytic function $P=P(u,v)$ satisfying the properties in Lemma \ref{sessantatre} depends only on one of the variables, that is, either $P=P(u)$ or $P=P(v)$. 
\end{theorem}

It follows that the same holds for the conformal factor $E$. We summarize:

\begin{cor} Let $\Sigma$ be a minimal surface in $M_{\sigma}$ satisfying the divergence condition \eqref{divcondition}. Assume that $\Sigma$ is not totally geodesic. Then in an isothermal parametrization with asymptotic line coordinates the conformal factor $E$ depends only on one of the variables, that is, either $E=E(u)$ or $E=E(v)$. 
\end{cor}


\begin{lemma} In the hypothesis of the previous corollary, every non-umbilical point has a neighborhood which is geodesically ruled. Since $\Sigma$ is complete the whole surface is foliated by geodesics of $M_{\sigma}$, hence, it is geodesically ruled. 
\end{lemma}

\begin{proof} Assume that $E=E(u)$, so that $E_v=0$. We claim that the lines $v=c$ are geodesics of $\Sigma$; since they are also asymptotic lines, they are necessarily geodesics also of the ambient $M_{\sigma}$. This gives the claimed foliation by geodesics. We now prove the claim. 
Consider the local orthonormal frame $(e_1,e_2)$ obtained by normalizing $(\bd_1,\bd_2)$:
$$
\twosystem
{e_1=\dfrac{1}{\sqrt{E}}\bd_1}
{e_2=\dfrac{1}{\sqrt{E}}\bd_2}
$$
Now,  the integral curves of $e_1$ are the asymptotic lines $v=c$, so that if $\gamma(t)$ is a unit speed parametrization of $v=c$, we have $\gamma'(t)=e_1$ for all $t$. Then, $v=c$ is a geodesic if and only if $\bar\nabla_{e_1}{e_1}=0$. Clearly, it is enough to show that
$$
\scal{\bar\nabla_{e_1}e_1}{e_2}=0.
$$
The Christoffel symbols of the parametrization are:
$$
\begin{aligned}
\Gamma^1_{11}&=\Gamma^2_{12}&=\dfrac{E_u}{2E}\quad \Gamma^2_{11}=-\dfrac{E_v}{2E}\\
\Gamma^2_{22}&=\Gamma^1_{12}&=\dfrac{E_v}{2E}\quad \Gamma^1_{22}=-\dfrac{E_u}{2E}
\end{aligned}
$$
and a straightforward verification gives:
$$
\scal{\bar\nabla_{e_1}e_1}{e_2}=\Big(\dfrac{1}{\sqrt{E}}\Big)_v=0
$$
as claimed.
\end{proof}

\begin{remark}
We observe that a similar calculation shows that the geodesic curvature (up to sign) of the line $u=c$ is given by
$$
\scal{\bar\nabla_{e_2}e_2}{e_1}=\Big(\dfrac{1}{\sqrt{E}}\Big)_u
$$
and hence it is constant along $u=c$ because $E_v=E_{uv}=0$.
\end{remark}


\section{End of proof}\label{Sec:Final}

Summarizing, let $\Omega$ be a domain with the $\frac 12$-property. Then its boundary is a minimal surface satisfying the divergence condition, and by the previous section it is geodesically ruled. To  complete the proof,
we give below the classification of the embedded, minimal, geodesically ruled surfaces in space-forms $M_{\sigma}$ with $\sigma\in\{0,1,-1\}$. Besides the totally geodesic surfaces, we have:

\begin{itemize}
\item in $\mathbb R^3$: the right helicoid (Catalan \cite{Ca42});

\item in $\mathbb H^3$: do Carmo and Dajczer \cite{dCD83}: we have the one-parameter family of minimal helicoids described in the introduction;

\item in $\mathbb S^3$: the Clifford torus. We explain  the argument in the next remark. 
\end{itemize}

\begin{remark}\label{Rmk_CliffordTorus}
In $\mathbb S^3$ the classification is due to Lawson \cite{La70}: there he shows that any ruled minimal surface is a {\it spherical helicoid}, a surface indexed by $\alpha>0$ and parametrized as follows:
\begin{align}
X^\alpha:(u,v) & \mapsto \left(\begin{array}{c}
\cos(\alpha u) \cos(v)\\ \sin(\alpha u) \cos(v)\\ \cos(u) \sin(v)\\ \sin(u) \sin(v)
\end{array} \right).
\end{align}
The Clifford torus corresponds to $\alpha=1$. Now, if $\alpha$ is irrational, then $X^{\alpha}$ is easily seen to be an injective immersion; to be proper,  the image must be  closed, hence compact. But then $X^{\alpha}$ can't be a homeomorphism onto its image.  Hence the immersion is not proper and the image is not the boundary of any domain. If $\alpha$ is rational then we recover the family of immersed minimal tori and $X^{\alpha}$ is an embedding only if $\alpha=1$, which corresponds to the Clifford torus.
\end{remark}
 
With this, the proof of the main theorem is complete.  

\smallskip


\bigskip

\textbf{Acknowledgments.} The authors would like to thank A. Grigor'yan, J. Masamune, M. Muratori and S. Pigola for useful suggestions and discussions about uniqueness results for the heat equation in stochastically complete Riemannian manifolds. The first author is a member of the Gruppo Nazionale per l'Analisi Matematica, la Probabilità e le loro Applicazioni (GNAMPA) of the Istituto Nazionale di Alta Matematica (INdAM). The second author is a member of the Gruppo Nazionale per le Strutture Algebriche, Geometriche e le loro Applicazioni
(GNSAGA) of the Istituto Nazionale di Alta Matematica (INdAM).


\appendix



\section{Dirichlet flow expansion in unbounded smooth domains}\label{Sec:AppendixC}

Let $(M,\scal{\cdot}{\cdot})$ be a geodesically complete Riemannian manifolds and $\Omega\subseteq M$ a smooth bounded domain. In \cite{Sa04} the second author managed to prove the following asymptotic expansion of the heat flow in the context of compact smooth domains.

\begin{theorem}\label{Thm:SavoExpansion}
Assume $\overline{\Omega}$ is compact. Let $u_0\in C^\infty(\overline \Omega)$ and consider the Dirichlet temperature $u_D$ on $\Omega$ with initial data $u_0$, as in \eqref{Eq:DirichletTemperature}.
For any $x\in \partial \Omega$ it holds
\begin{align}
\frac{\partial u_D}{\partial N} (t,x) \sim \frac{u_0(x)}{\sqrt{\pi}} t^{-\frac{1}{2}}+ \sum_{k=0}^{\infty} \gamma_k(x) t^{\frac{k}{2}} \quad as\ t\to 0
\end{align}
for certain invariants $\gamma_k\in C^{\infty}(\Omega)$ that depend on $u_0$ and the extrinsic geometry of the boundary.
\end{theorem}

In this appendix we extend the previous result to the setting of noncompact domains. We give an idea of the proof. First, we observe that the expansion is purely local, that is, it depends 
only on the behavior of $u_0$ and the geometry near the given point $x$. In fact, as recalled in \eqref{Eq_Gamma_k}, one has, for all $k$:
$$
\gamma_k(x)=\left(1+\frac k2 \right)D_{k+2}u_0(x),
$$
where $D_{k+2}$ is a certain differential operator acting on $C^{\infty}(U)$, where $U$ is an arbitrary small neighborhood of $x$ in $\Omega$ (we refer to \cite{Sa04} for the recursive formula computing the operators $D_n$). The second observation is that the asymptotic development is polynomial in $\sqrt t$: hence, any change in the domain which occours far from a given point $\bar x$ and which produces only an exponentially decreasing effect as $t\to 0$ will not affect the expansion at $\bar x$. 

\smallskip

We proceed as follows. Let $\Omega$ be a (possibly unbounded) smooth domain, fix $\bar x\in\bd\Omega$ and consider a bounded open set $V\subseteq\Omega$ with smooth boundary $\bd V$ such that:

\begin{enumerate}
\item[1.] $\bd V\cap\bd\Omega$ contains an open neighbourhood $\Gamma$ of $\bar x$ in $\bd\Omega$,
\item[2.] $d^M(\bar x,\bd V\cap\Omega)\geq c$ for a given positive constant $c$.
\end{enumerate}
\begin{figure}[h!]
\centering
\includegraphics[width=13cm]{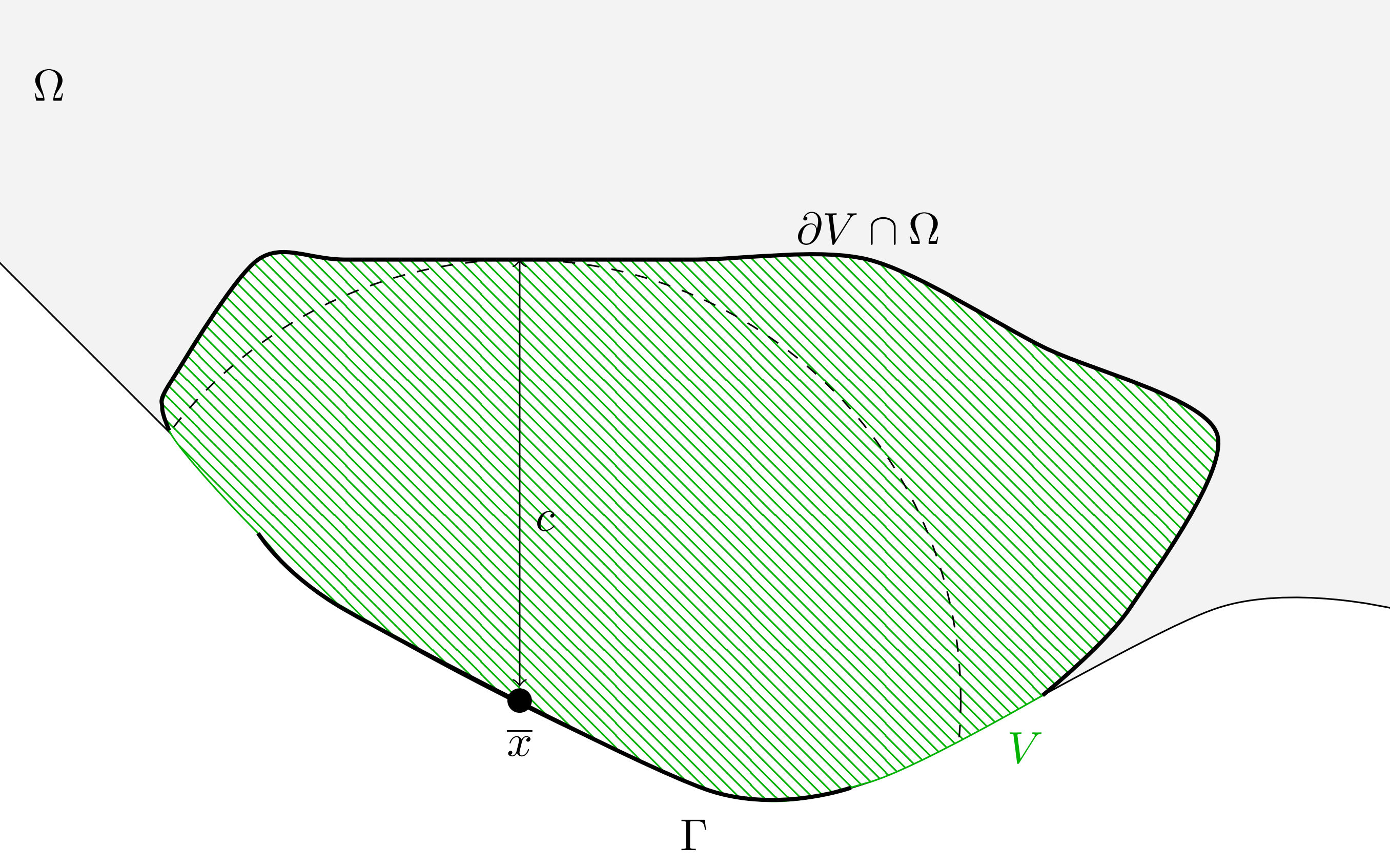}
\end{figure}
Let  $u_{\Omega}(t,x)$ (resp. $u_V(t,x)$) be the solution of the heat equation on $\Omega$ (resp. on $V$) with Dirichlet boundary conditions and initial data $u_0\in C^{\infty}(\Omega)$ (resp. 
${u_0}_{\vert V}\in C^{\infty}(V)$). The scope is to show that
\begin{equation}\label{otemme}
\derive{u_{\Omega}}{N}(t,\bar x)-\derive{u_{V}}{N}(t,\bar x)=O(t^m)\quad\text{for all $m\in\mathbb N$}
\end{equation}
Once this is proved, we see that $\derive{u_{\Omega}}{N}(t,\bar x)$ admits an asymptotic expansions which coincides, at any point $\bar x\in\bd\Omega$,  with that of $\derive{u_{V}}{N}(t,\bar x)$, and the proof is complete.

\medskip

We then show \eqref{otemme}. We first recall Duhamel principle, which provides a representation formula for smooth functions defined over a compact manifold in terms of any heat kernel (fundamental solution of the heat equation).

\begin{proposition}[Compact Duhamel Principle]\label{Prop_CompactDuhamelPrinciple}
Let $M$ be a compact Riemannian manifold with nonempty boundary $\bd M$, $p(t,x,y)$ a heat kernel on $M$ and $w(t,x)$ a smooth function on $[0,\infty)\times M$.

Then, for every $(t,x)\in [0,\infty)\times M$
\begin{align}
\begin{split}
w(t,x)= & \int_M p(t,x,y)\ w(0,y) \dvol_y -\int_0^t \int_{\partial M} p(t-\tau, x, y)\ \frac{\partial w}{\partial N_y}(\tau,y) \dsup_y \dint{\tau}\\
& + \int_0^t \int_{\partial M} \frac{\partial p}{\partial N_y}(t-s, x, y)\ w(s,y) \dsup_y \dint{s}\\
& + \int_0^t \int_M p(t-s,x,y)\ \left( \frac{\partial}{\partial t}+\Delta\right)w(s,y) \dvol_y \dint{s},
\end{split}
\end{align}
where $N$ is the inner normal field to $\partial M$. 

\end{proposition}
\begin{proof}: integration by parts. \end{proof}

We now take $M=V$ and $p=k_V$, the Dirichlet heat kernel of $V$. Choosing $w=u_{\Omega}$ we see that  $k_V$ vanishes on $\bd V$ and $\left( \frac{\partial}{\partial t}+\Delta \right)w=0$. Hence, at $x\in V$ we have:
$$
\begin{aligned}
u_{\Omega}(t, x)&=\int_V k_V(t, x,y)u_0(t,y)\dvol_y+\int_0^t \int_{\bd V} \frac{\partial k_V}{\partial N_y}(t-s, x, y)\ u_{\Omega}(s,y) \dsup_y \dint{s}.
\end{aligned}
$$
Since the restriction of $u_{\Omega}$ to $\bd V$ is supported on $A\doteq \overline{\Omega\cap\bd V}$, we see that:
$$
u_{\Omega}(t, x)-u_V(t, x)=\int_0^t \int_{A} \frac{\partial k_V}{\partial N_y}(t-s, x, y)\ u_{\Omega}(s,y) \dsup_y \dint{s}
$$
Fix $\bar x\in\bd V\cap\bd\Omega$ as above. Since $A$ is at distance $c>0$ from $\bar x$, we see that the integrand in the right-hand side of previous equation is $C^\infty$ on $[0,t]\times A$ for all $x$ sufficiently close to $\overline{x}$. We can then differentiate under the integral sign in the normal direction from $\overline{x}$ and obtain
\begin{align}
\frac{\partial u_{\Omega}}{\partial N_{\overline{x}}}(t,\overline{x})-\frac{\partial u_V}{\partial N_{\overline{x}}}(t,\overline{x}) = \int_0^t \int_{A} \frac{\partial}{\partial N_{\overline{x}}}\frac{\partial k_V}{\partial N_y}(t-s,\overline{x},y)\ u_\Omega (s,y) \dsup_y \dint{s}.
\end{align}
Now, since $d^M(\overline{x},y)>c>0$ for every $y\in A$, since the Dirichlet heat kernel $k_D$ is rapidly decaying off the diagonal (see \cite{Ch84, Grie04}), we see that
\begin{align}
\frac{\partial}{\partial N_{\overline{x}}}\frac{\partial k_V}{\partial N_y}(t-s,\overline{x},y)=O((t-s)^m)
\end{align}
for every $m\in \mathbb N$. Since $u_\Omega 
(s,y)=O(1)$, we get the final claim \eqref{otemme}.


\section{Proof of theorem \ref{Thm_UniquenessDirichletAppendix}}

\subsection{Proof of Step 1}\label{Sec:Rigidity_Uniqueness}
Let $(M,\scal{\cdot}{\cdot})$ be an arbitrary smooth Riemannian manifolds and let $p$ be its unique minimal positive heat kernel. We recall that $M$ is said to be \textit{stochastically complete} if for all $y\in M$ and for all $t>0$ 
\begin{align}
    \int_M p(t,x,y) \dvol_x=1.
\end{align}
By the minimality $p$, and the fact that for any fundamental solution $q$ of the heat equation one has
\begin{align}
    \int_M q(t,x,y) \dvol_x\leq 1 \quad \forall y\in M\ \forall t>0,
\end{align}
it follows that if $(M,\scal{\cdot}{\cdot})$ is stochastically complete, then every fundamental solution of the heat equation coincides $p$.

Several characterizations of stochastic completeness are present in literature, involving analytical and geometric conditions and using both analytical and probabilistic approaches, see for instance \cite{alias2016maximum, BM23, Gr09, Gr99, GIM20, GIMP23, PRS03, PRS05}. Some of them involve uniqueness results for both elliptic and parabolic PDEs settled on the whole manifold. For the purpose of the present work, we are interested in the next two characterizations. The first one involves the Cauchy problem for the heat equation (see \cite{Gr99, GIM20}).

\begin{theorem}\label{Thm_StandardSCUniqueness}
A Riemannian manifold $(M,g)$ is stochastically complete if and only if the constant null function is the unique bounded solution to the Cauchy problem
    \begin{align}
        \begin{cases}
            \left(\frac{\partial}{\partial t}+\Delta\right) u=0 & in\ (0,+\infty)\times M\\ u(0,\cdot)=0 & on\ M.
        \end{cases}
    \end{align}
\end{theorem}

The second characterization is given in terms of elliptic differential equations and states as follows (for a reference see \cite[Theorem 2.14]{alias2016maximum} and the comments below).

\begin{theorem}[\cite{alias2016maximum}]\label{Thm:SCUniqueness}
    A Riemannian manifold $(M,\scal{\cdot}{\cdot})$ is stochastically complete if and only if for every $\lambda>0$ the only bounded, nonnegative and $C^0(M)\cap W^{1,2}_{loc}(M)$ distributional solution to $\Delta U\leq -\lambda U$ is the constant null function.
\end{theorem}

\begin{proof}[Proof of Theorem \ref{Thm_UniquenessDirichletAppendix}]
    Let $v$ be as in the assumptions of the theorem. Define
    \begin{align}
        V_\lambda :\overline{\Omega}&\to \rr\\
        x&\mapsto \int_0^{+\infty} e^{-\lambda t} v(t,x) \dint{t},
    \end{align}
    where $\lambda\in (0,+\infty)$. We proceed by steps.

    \textbf{1}: \underline{$\Delta V_\lambda= -\lambda V_\lambda$ in $\Omega$}. First step consists in proving that $\Delta V_\lambda=-\lambda V_\lambda$ in the sense of distributions. To this aim, firstly observe that for any $\lambda>0$
    \begin{align}\label{Eq_Distributional_Appendix0}
        \left(\frac{\partial}{\partial t} + \Delta\right) (e^{-\lambda t} v) = -\lambda e^{-\lambda t} v
    \end{align}
    in the sense of distributions. Now fix a sequence of test functions $\Psi_{n,m}(t,x)=\psi(x)\phi_{n,m}(t)$ where $0\leq \psi\in C^\infty_c(\Omega)$ and $0\leq \phi_{n,m}\in C^{0,1}(0,+\infty)$ is given by
    \begin{align}
        \phi_{n,m}(t)=\begin{cases}
            0 & \inn \left[0,\frac{1}{n}\right)\\
            (tn-1)\left(\frac{2m}{n} \right)^{1/2} & \inn \left[\frac{1}{n},\frac{2}{n}\right)\\
            (mt)^{1/2} & \inn \left[\frac{2}{n}, \frac{1}{m}\right)\\
            1 & \inn \left[\frac{1}{m},m-1\right)\\
            m-t & \inn \left[m-1,m\right)\\
            0 & \inn \left[m,+\infty\right)
        \end{cases}
    \end{align}
    for $m<\frac{n}{2}$. In particular, the function $\frac{\partial \phi_{n,m}}{\partial t}$ exists almost everywhere and it is equal to
    \begin{align}
        \frac{\partial \phi_{n,m}}{\partial t}(t)=\begin{cases}
            0 & \inn \left(0,\frac{1}{n}\right)\\
            \left(2mn \right)^{1/2} & \inn \left(\frac{1}{n},\frac{2}{n}\right)\\
            \frac{1}{2} \left(\frac{m}{t}\right)^{1/2} & \inn \left(\frac{2}{n}, \frac{1}{m}\right)\\
            0 & \inn \left(\frac{1}{m},m-1\right)\\
            -1 & \inn \left(m-1,m\right)\\
            0 & \inn \left(m,+\infty\right).
        \end{cases}
    \end{align}
    By testing equation \eqref{Eq_Distributional_Appendix0} (in the distributional sense) with $\Psi_{n,m}$, we get
    \begin{align}
        - \lambda \int_\Omega \int_0^{+\infty} & e^{-\lambda t} v(t,x) \Psi_{n,m}(t,x) \dint{t} \dvol_x \\
        &=\int_\Omega \int_0^{+\infty} e^{-\lambda t} v(t,x) \left( -\frac{\partial}{\partial t} + \Delta \right) \Psi_{n,m}(t,x) \dint{t} \dvol_x\\
        &=\int_\Omega \int_0^{+\infty} e^{-\lambda t} v(t,x) \phi_{n,m}(t) \Delta \psi(x) \dint{t} \dvol_x\\
        &\quad - \int_\Omega \int_0^{+\infty} e^{-\lambda t} v(t,x) \psi(x) \frac{\partial \phi_{n,m}(t)}{\partial t} \dint{t} \dvol_x.
    \end{align}
    By dominated convergence theorem, letting first $n$ go to $+\infty$ and then $m$ to $+\infty$, the first integral converges
    \begin{align}
        \int_\Omega \int_0^{+\infty} & e^{-\lambda t} v(t,x) \phi_{n,m}(t) \Delta \psi(x) \dint{t} \dvol_x\\
        &\xrightarrow[m\to +\infty]{n\to +\infty} \int_\Omega \int_0^{+\infty} e^{-\lambda t} v(t,x) \dint{t}\ \Delta \psi(x) \dvol_x.
    \end{align}
    Analogously
    \begin{align}
        - \lambda \int_\Omega \int_0^{+\infty} & e^{-\lambda t} v(t,x) \Psi_{n,m}(t,x) \dint{t} \dvol_x\\
        &\xrightarrow[m\to +\infty]{n\to +\infty}  - \lambda \int_\Omega \int_0^{+\infty} e^{-\lambda t} v(t,x) \dint{t}\ \psi(x) \dvol_x.
    \end{align}
    To conclude, we observe that
    \begin{align}
        \Bigg|\int_\Omega \int_0^{+\infty}& e^{-\lambda t} v(t,x) \psi(x) \frac{\partial \phi_{n,m}(t)}{\partial t} \dint{t} \dvol_x\Bigg| \\
        &=\Bigg|\int_\Omega \left[\int_{\frac{1}{n}}^{\frac{2}{n}} e^{-\lambda t} v(t,x) (2mn)^{1/2} \dint{t}+\int_{\frac{2}{n}}^{\frac{1}{m}} e^{-\lambda t} v(t,x) \frac{m^{1/2}}{2t^{1/2}} \dint{t}\right.\\
        &\quad \quad \quad \left. - \int_{m-1}^{m} e^{-\lambda t} v(t,x) \dint{t}\right] \psi(x) \dvol_x\Bigg| \\
        &\leq \int_\Omega \left[||v||_{L^\infty((0,+\infty)\times \Omega)}(2mn)^{1/2} \frac{1}{n} + \int_{\frac{2}{n}}^{\frac{1}{m}} e^{-\lambda t} |v(t,x)| \frac{m^{1/2}}{2t^{1/2}} \dint{t}\right.\\
        &\quad \quad \quad \left. + \int_{m-1}^{m} e^{-\lambda t} |v(t,x)| \dint{t}\right] \psi(x) \dvol_x\\
        \xrightarrow[]{n\to +\infty}&\int_\Omega \left[\int_{0}^{\frac{1}{m}} e^{-\lambda t} |v(t,x)| \frac{m^{1/2}}{2t^{1/2}} \dint{t}+ \int_{m-1}^{m} e^{-\lambda t} |v(t,x)| \dint{t}\right] \psi(x) \dvol_x\\
        &\leq \int_\Omega \left[\sup_{t\in [0,1/m]} |v(\cdot,x)| + \int_{m-1}^{m} e^{-\lambda t} |v(t,x)| \dint{t}\right] \psi(x) \dvol_x\\
        &\xrightarrow[]{m\to +\infty} 0
    \end{align}
    by dominated convergence theorem, and recalling that $v(t,x)\to 0$ as $t\to 0$ locally uniformly in $x\in \Omega$, thank to the continuity of $v$ in $[0,+\infty)\times \Omega$. It follows that for any $\lambda>0$ the smooth function $V_\lambda$ satisfies
    \begin{align}
            \int_\Omega V_\lambda(x) \Delta\psi(x) \dvol_x=-\lambda \int_\Omega V_\lambda (x) \psi(x) \dvol_x \quad \quad \forall 0\leq \psi \in C^\infty_c(\Omega)
    \end{align}
    and $V_\lambda \equiv 0$ on $\partial \Omega$. Hence, $V_\lambda$ is a distributional (and hence classical) solution to
    \begin{align}
        \begin{cases}
            \Delta V_\lambda=-\lambda V_\lambda & \inn \Omega\\
            V_\lambda=0 & \onn \partial \Omega.
        \end{cases}
    \end{align}
    
    \textbf{2}: \underline{$V_\lambda$ is the constant null function}. Fix $\e>0$ and observe that 
    \begin{align}
    \Delta(V_\lambda-\e)= -\lambda V_\lambda \leq -\lambda (V_\lambda-\e).
    \end{align}
    By Brezis-Kato's inequality (see \cite[Proposition 4.1]{PVV24}) we have
    \begin{align}
            \Delta (V_\lambda-\e)_+ \leq -\lambda (V_\lambda-\e)_+ \quad \inn \Omega.
    \end{align}
    In particular, $(V_\lambda-\e)_+\in C^0(\Omega) \cap W^{1,2}_{loc}(\Omega)$ by \cite[Corollary 4.3]{PVV24}. Now extend $(V_\lambda-\e)_+$ to $0$ outside $\Omega$ and denote by $W$ this extension: then $0\leq W\in L^\infty(\Omega)\cap C^0(M) \cap W^{1,2}_{loc}(M)$ satisfies $\Delta W \leq -\lambda W$ in $M$ in the sense of distributions. By the stochastic completeness of $M$ (see Theorem \ref{Thm:SCUniqueness}) it follows that $W\equiv 0$ in $M$, implying $(V_\lambda-\e)_+\equiv 0$ in $\Omega$ and hence $V_\lambda\leq \e \quad \inn \Omega$. Since it holds for every $\e>0$, it follows that $V_\lambda\leq 0$ in $\Omega$. By replacing $V_\lambda$ with $-V_\lambda$, we get $V_\lambda\equiv 0$.

    \textbf{3}: \underline{Conclusion}. To conclude, we just observe two facts. The first, is that for any $x\in \Omega$ the function
    \begin{align}
         v(\cdot,x):[0,+\infty)&\to \rr\\
         t&\mapsto v(t,x).
    \end{align}
    is bounded and continuous. Secondly, by the second step it follows that for any fixed $x\in \Omega$ the function 
    \begin{align}
        \mathcal{L}[v(\cdot,x)]:(0,+\infty)&\to \rr \\
        \lambda &\mapsto V_\lambda(x)=\int_0^{+\infty} e^{-\lambda t} v(t,x) \dint{t},
    \end{align}
    which is the Laplace transform of $v(\cdot, x)$, is the constant null function $\mathcal{L}[v(\cdot,x)]\equiv 0$. By the standard theory of the Laplace transform (see \cite[Theorem 1.23]{schiff2013laplace}), it follows that $v(\cdot,x)\equiv 0$. Since it holds for any $x\in \Omega$, we get the claim.
\end{proof}


\section{Proof of Theorem \ref{Thm_oneofthevariables}}

We first prove the following general fact.

\begin{theorem}\label{Thm_General_oneofthevariables} Let $P=P(u,v)$ be a positive analytic function defined on an open set $\Omega$ of the plane containing the origin. Assume that $P(0,0)=1$, $P_{uv}=0$ on $\Omega$ and 
\begin{equation}\label{cond}
\Delta\log P=\phi(P),
\end{equation}
for an analytic function $\phi=\phi(r)$. Set $\psi(r)=r^2\phi(r)$ and assume further that the following condition holds:
\begin{align}\label{Eq:DerivativesConditionPsi}
\psi''(1)-2\psi'(1)\ne 0.
\end{align}
Then $P$ depends only on one of the two variables: $P(u,v)=1+f(u)$ or $P(u,v)=1+g(v)$.
\end{theorem} 

For the proof, we start by observing that, since $P_{uv}=0$, we can write:
$$
P(u,v)=1+f(u)+g(v),
$$
for analytic functions $f,g$ of one real variable defined in a neighborhood of the origin. Condition \eqref{cond} is equivalent to
$$
P^2\Delta\log P=\psi(P).
$$
Since $P^2\Delta\log P=\abs{\nabla P}^2+P\Delta P$ this identity is rewritten:
\begin{equation}\label{mainidentity}
f'^2+g'^2-(1+f+g)(f''+g'')=\psi(1+f+g).
\end{equation}
We develop in Taylor series at the origin and set:
$$
\begin{aligned}
f(u)&=a_1u+a_2u^2+a_3u^3+a_4u^4+a_5u^5+\dots\\
g(v)&=b_1v+b_2v^2+b_3v^3+b_4v^4+b_5v^5+\dots\\
\psi(1+f+g) & =F_{00}+F_{10}u+F_{01}v+F_{20}u^2+F_{11}uv+F_{02}v^2+\dots
\end{aligned}
$$
We set $\Gamma_k=\psi^{(k)}(1)$; a straightforward calculation yields the following identities:

\begin{equation}\label{zero}
a_1^2+b_1^2-2(a_2+b_2)=\Gamma_0
\end{equation}
\begin{equation}\label{onet}
2a_1(a_2-b_2)-6a_3=\Gamma_1a_1
\end{equation}
\begin{equation}\label{two}
2b_1(b_2-a_2)-6b_3=\Gamma_1b_1
\end{equation}
\begin{equation}\label{three}
2a_2^2-12a_4-2a_2b_2=a_2\Gamma_1+\frac 12a_1^2\Gamma_2
\end{equation}
\begin{equation}\label{four}
2b_2^2-12b_4-2a_2b_2=b_2\Gamma_1+\frac 12b_1^2\Gamma_2
\end{equation}
\begin{equation}\label{five}
-6a_1b_3-6a_3b_1=a_1b_1\Gamma_2
\end{equation}
\begin{equation}\label{six}
-12a_4b_1-6a_2b_3=\frac 12\Big(\Gamma_3a_1^2b_1+2\Gamma_2a_2b_1\Big)
\end{equation}
\begin{equation}\label{seven}
-12a_1b_4-6a_3b_2=\frac 12\Big(\Gamma_3b_1^2a_1+2\Gamma_2b_2a_1\Big).
\end{equation}
\begin{equation}\label{eight}
\text{if $a_1=b_1=0$:}\quad -12(a_2b_4+a_4b_2)=\Gamma_2a_2b_2 
\end{equation}

From this table we get the following set of relations.

\begin{lemma} The following identities hold:

\begin{equation}\label{idone}
(\Gamma_2-2\Gamma_1)a_1b_1=0
\end{equation}
\begin{equation}\label{idtwo}
b_1\Big((\Gamma_3-\Gamma_2)a_1^2+2(\Gamma_2-2\Gamma_1)a_2\Big)=0
\end{equation}
\begin{equation}\label{idthree}
a_1\Big((\Gamma_3-\Gamma_2)b_1^2+2(\Gamma_2-2\Gamma_1)b_2\Big)=0
\end{equation}
\begin{equation}\label{idfour}
(\Gamma_2-2\Gamma_1)a_2b_2=0\quad\text{if $a_1=b_1=0$}
\end{equation}
\end{lemma}

\begin{proof} Identity \eqref{idone} is obtained considering \eqref{onet}, \eqref{two} and \eqref{five}; identity \eqref{idtwo} is obtained considering \eqref{two}, \eqref{three} and \eqref{six}; identity \eqref{idthree} is the mirror case; identity \eqref{idfour} is obtained considering \eqref{three}, \eqref{four} and \eqref{eight}.
\end{proof}

Recall our assumption $\Gamma_2-2\Gamma_1\ne 0$.
From \eqref{idone} one sees that $a_1b_1=0$.

\subsection{First case:  $b_1=0$.} Then $b_3=0$ from \eqref{two} and $a_1b_2=0$ from \eqref{idthree}.

\smallskip

{\bf Subcase $b_1=b_2=0$.}  Set $u=0$ so that $P(0,v)=1+g(v)$; the differential equation for $g$ takes the form:
$$
a_1^2+g'^2-(1+g)(2a_2+g'')=\psi(1+g).
$$
One sees that then $g(0)=g'(0)=0$; since by \eqref{zero} we have $a_1^2-2a_2=\psi(1)$, we see that $g(v)=0$ is the only solution; hence $P(u,v)=1+f(u)$ depends only on $u$.

\smallskip

{\bf Subcase $b_1=0, b_2\ne 0$.} Then $a_3=0$ from \eqref{onet} and $a_1=0$ from \eqref{idthree}. As $b_2\ne 0$ and $a_1=b_1=0$ we also get $a_2=0$ from \eqref{idfour}.
Now set $v=0$, that is, consider the function $f(u)=P(u,0)$. It satisfies the differential equation:
$$
f'^2-(1+f)(2b_2+f'')=\psi(1+f),
$$
and we know that $f(0)=f'(0)=0$: hence $f(u)=0$ and $P(u,v)=1+g(v)$.

\subsection{Second case:  $a_1=0$.} It is the mirror case, and is discussed in the same way, by replacing the $a$'s with the $b$'s. 
 
\smallskip 

The proof of Theorem \ref{Thm_General_oneofthevariables}  is complete. We can now prove Theorem \ref{Thm_oneofthevariables}.

\begin{proof}[Proof of Theorem \ref{Thm_oneofthevariables}]
We apply Theorem \ref{Thm_General_oneofthevariables} when 
$\phi(x)=6kx^{1/3}-\frac{6\sigma}{k}x^{-1/3}$
that is,
$$
\psi(x)=6kx^{7/3}-\frac{6\sigma}{k}x^{5/3}.
$$
The condition \eqref{Eq:DerivativesConditionPsi} translates into the fact that:
$$
k^2=\dfrac 56\sigma.
$$
This is impossible if $\sigma=0,-1$. If $\sigma=1$ this would imply (recall that $k=E(0,0)^{-1}$) that 
$$
E(0,0)^2=\dfrac 65.
$$
We can assume that $E$ is not constant on $\Omega$; hence we can always put the origin at a point where the above equality does not hold, and the conclusion follows. The proof is complete.
\end{proof}

\begin{remark}
Note that condition \eqref{Eq:DerivativesConditionPsi} is crucial for the proof of Theorem \ref{Thm_oneofthevariables}. To see why, consider the function $P(u,v)=1+u^2+v^2$, which depends both on $u$ and $v$. This function satisfies
\begin{align}
    \Delta \log(P)=\phi(P),
\end{align}
where $\phi(r)=4r^{-2}$ is analytic around $r=1$, and $\psi(r)=r^2 \phi(r)=4$. In this case $\psi''-2\psi'\equiv 0$, so condition \eqref{Eq:DerivativesConditionPsi} cannot be satisfied at any point. This example illustrates that the condition is not merely technical, but necessary for the result to hold.
\end{remark}

\section{Proof of Theorem \ref{Nitscheg}}\label{Appendix:Nitsche}

We start by supposing that $\Omega$ is $\frac{1}{2}$-uniformly dense in $\partial \Omega$. We recall that if $(M,g)$ is a Riemannian manifold with associated heat kernel $k$ and $\Phi\in \isom{M}$, then
\begin{align}
k(t,p,q)=k(t,\phi(p),\phi(q)) \quad \forall(t,p,q)\in \rr_{>0}\times M\times M.
\end{align}
For a reference see \cite{Gr09}. In particular, it follows that the heat kernel of any model manifold $M_\sigma^m$ only depends on the geodesic distance from a fixed pole $o$. Hence, the heat kernels of $\rr^3, \sss^3$ and $\hh^3$ are radial. Then, we can apply the argument presented in Section \ref{uniformdensity}, obtaining that $u_C\equiv \frac{1}{2}$ on $(0,+\infty)\times \partial \Omega$.

Conversely, if $\Omega$ is a $\frac{1}{2}$-domain, then its boundary is one of the surfaces listed in Theorem \ref{Thm_Main}. As showed in Section \ref{Subsec:DensitySigmaConstant}, using the same notation, there exist a subgroup of isometries $\mathcal{T}\leq \isom M$ and an isometry $\Psi\in \isom M$ such that
 \begin{enumerate}
        \item every $T\in\mathcal{T}$ preserves both $\Omega_+$ and $\Omega_-$, hence $T(\partial \Omega)=\partial \Omega$;
        \item $\Psi$ sends $\Omega_+$ to $\Omega_-$ and viceversa;
        \item\label{3_Lem_12_RototranslationsReflection} for every $x\in \partial \Omega$ there exists $T\in \mathcal{T}$ such that $\Psi(x)=T(x)$.
\end{enumerate}
Fix  $x\in \partial \Omega$ and $R>0$ and let $T\in \mathcal{T}$ be such that $T(x)=\Psi(x)$: since both $T$ and $\Psi$ are isometries, by the properties listed above it is easy to see that
\begin{align}
    |B_R(x)\cap \Omega_+|&= |\Psi(B_R(x) \cap \Omega_+)|\\
    &=|B_R(\Psi(x))\cap \Omega_-|\\
    &=|B_R(T(x))\cap \Omega_-|\\
    &=|T(B_R(x) \cap \Omega_-)|\\
    &=|B_R(x) \cap \Omega_-|.
\end{align}
It follows that for every $x\in \partial \Omega$ and for every $R>0$
\begin{align}
    |B_R(x)|=|B_R(x)\cap \Omega_+|+|B_R(x)\cap \Omega_-|=2 |B_R(x)\cap \Omega_+|.
\end{align}
By the Lebesgue's differentiation theorem, it follows that for every $x\in \partial \Omega$
\begin{align}
    |\partial B_R(x)|=2 |\partial B_R(x)\cap \Omega_+|
\end{align}
for almost every $R>0$, obtaining the claim.

\begin{remark}
    As can be noticed in the above proof, Theorem \ref{Lem_12_RototranslationsReflection} is, in fact, a sufficient condition for a domain to be $\frac{1}{2}$-uniformly dense in its boundary.
\end{remark}

\bibliographystyle{abbrv}
\bibliography{references}

\begin{thebibliography}{10}

\bibitem{Al90}
G.~Alessandrini.
\newblock Matzoh ball soup: a symmetry result for the heat equation.
\newblock {\em Journal d’Analyse Math{\'e}matique}, 54(1):229--236, 1990.

\bibitem{alias2016maximum}
L.~J. Al{\'\i}as, P.~Mastrolia, M.~Rigoli, et~al.
\newblock {\em Maximum principles and geometric applications}.
\newblock Springer, 2016.

\bibitem{BM23}
A.~Bisterzo and L.~Marini.
\newblock The ${L}^\infty$-positivity preserving property and stochastic
  completeness.
\newblock {\em Potential Analysis}, 59(4):2017--2034, 2023.

\bibitem{dCD83}
M.~D. Carmo and M.~Dajczer.
\newblock Rotation hypersurfaces in spaces of constant curvature.
\newblock {\em Transactions of the American Mathematical Society}, pages
  685--709, 1983.

\bibitem{Ca42}
E.~Catalan.
\newblock Sur les surface r{\'e}gl{\'e}es dont l'aire est un minimum.
\newblock {\em Journal de Math{\'e}matiques Pures et Appliqu{\'e}es},
  7:203--211, 1842.

\bibitem{CSU23}
L.~Cavallina, S.~Sakaguchi, and S.~Udagawa.
\newblock A characterization of a hyperplane in two-phase heat conductors.
\newblock {\em Communications in Analysis and Geometry}, (7):1867--1888, 2023.

\bibitem{Ch84}
I.~Chavel.
\newblock {\em Eigenvalues in Riemannian geometry}.
\newblock Academic press, 1984.

\bibitem{Ci}
G.~Cimmino.
\newblock Sulla curvatura media delle superficie.
\newblock {\em Rendiconti del Circolo Matematico di Palermo (1884-1940)},
  56(1):281--288, 1932.

\bibitem{Grie04}
D.~Grieser.
\newblock Notes on heat kernel asymptotics.
\newblock {\em Available on his website: http://www. staff. unioldenburg.
  de/daniel. grieser/wwwlehre/Schriebe/heat. pdf}, 2004.

\bibitem{Gr09}
A.~Grigor'yan.
\newblock {\em Heat kernel and analysis on manifolds}, volume~47.
\newblock American Mathematical Soc., 2009.

\bibitem{Gr99}
A.~Grigor’yan.
\newblock Analytic and geometric background of recurrence and non-explosion of
  the brownian motion on riemannian manifolds.
\newblock {\em Bulletin of the American Mathematical Society}, 36(2):135--249,
  1999.

\bibitem{GIM20}
G.~Grillo, K.~Ishige, and M.~Muratori.
\newblock Nonlinear characterizations of stochastic completeness.
\newblock {\em Journal de Math{\'e}matiques Pures et Appliqu{\'e}es},
  139:63--82, 2020.

\bibitem{GIMP23}
G.~Grillo, K.~Ishige, M.~Muratori, and F.~Punzo.
\newblock A general nonlinear characterization of stochastic incompleteness.
\newblock {\em arXiv preprint arXiv:2301.07942}, 2023.

\bibitem{Hs02}
E.~P. Hsu.
\newblock {\em Stochastic analysis on manifolds}.
\newblock Number~38. American Mathematical Soc., 2002.

\bibitem{Kl64}
M.~Klamkin.
\newblock A physical characterization of a sphere.
\newblock {\em SIAM Review}, 6(1):61, 1964.

\bibitem{La70}
H.~B. Lawson.
\newblock Complete minimal surfaces in $\mathbb{S}^3$.
\newblock {\em Annals of Mathematics}, 92(3):335--374, 1970.

\bibitem{MPSS16}
R.~Magnanini, D.~Peralta-Salas, and S.~Sakaguchi.
\newblock Stationary isothermic surfaces in euclidean 3-space.
\newblock {\em Mathematische Annalen}, 364(1):97--124, 2016.

\bibitem{MPS06}
R.~Magnanini, J.~Prajapat, and S.~Sakaguchi.
\newblock Stationary isothermic surfaces and uniformly dense domains.
\newblock {\em Transactions of the American Mathematical Society},
  358(11):4821--4841, 2006.

\bibitem{MS02}
R.~Magnanini and S.~Sakaguchi.
\newblock Matzoh ball soup: heat conductors with a stationary isothermic
  surface.
\newblock {\em Annals of mathematics}, pages 931--946, 2002.

\bibitem{Mor81}
H.~Mori.
\newblock Minimal surfaces of revolution in $\mathbb{H}^3$ and their global
  stability.
\newblock {\em Indiana University Mathematics Journal}, 30(5):787--794, 1981.

\bibitem{Mor82}
H.~Mori.
\newblock On surfaces of right helicoid type in $\mathbb{H}^3$.
\newblock {\em Boletim da Sociedade Brasileira de
  Matem{\'a}tica-Bulletin/Brazilian Mathematical Society}, 13(2):57--62, 1982.

\bibitem{Ni95}
J.~C. Nitsche.
\newblock Characterizations of the mean curvature and a problem of {G}.
  {C}immino.
\newblock {\em Analysis}, 15(3):233--246, 1995.

\bibitem{Ni89}
J.~C.~C. Nitsche.
\newblock {\em Lectures on minimal surfaces. {V}ol. 1}.
\newblock Cambridge University Press, Cambridge, 1989.
\newblock Introduction, fundamentals, geometry and basic boundary value
  problems, Translated from the German by Jerry M. Feinberg, With a German
  foreword.

\bibitem{OR78}
O.~A. Ole\u~inik and E.~V. Radkevi\v~c.
\newblock The method of introducing a parameter for the investigation of
  evolution equations.
\newblock {\em Uspekhi Mat. Nauk}, 33(5(203)):7--76, 237, 1978.

\bibitem{PRS03}
S.~Pigola, M.~Rigoli, and A.~Setti.
\newblock A remark on the maximum principle and stochastic completeness.
\newblock {\em Proceedings of the American Mathematical Society},
  131(4):1283--1288, 2003.

\bibitem{PRS05}
S.~Pigola, M.~Rigoli, and A.~G. Setti.
\newblock {\em Maximum principles on Riemannian manifolds and applications}.
\newblock American Mathematical Soc., 2005.

\bibitem{PVV24}
S.~Pigola, D.~Valtorta, and G.~Veronelli.
\newblock Approximation, regularity and positivity preservation on riemannian
  manifolds.
\newblock {\em Nonlinear Analysis}, 245:113570, 2024.

\bibitem{Sak20}
S.~Sakaguchi.
\newblock Some characterizations of parallel hyperplanes in multi-layered heat
  conductors.
\newblock {\em Journal de Math{\'e}matiques Pures et Appliqu{\'e}es},
  140:185--210, 2020.

\bibitem{Sak24}
S.~Sakaguchi.
\newblock Interaction between initial behavior of temperature and the mean
  curvature of the interface in two-phase heat conductors.
\newblock {\em arXiv preprint arXiv:2408.15539}, 2024.

\bibitem{Sa01}
A.~Savo.
\newblock On the asymptotic series of the heat content.
\newblock {\em Contemp. Math.}, 288:428--432, 2001.

\bibitem{Sa04}
A.~Savo.
\newblock Asymptotics of the heat flow on a manifold with smooth boundary.
\newblock {\em Communications in Analysis and Geometry}, 12(3):671--702, 2004.

\bibitem{Sa16}
A.~Savo.
\newblock Heat flow, heat content and the isoparametric property.
\newblock {\em Mathematische Annalen}, 366(3):1089--1136, 2016.

\bibitem{Sa18}
A.~Savo.
\newblock Geometric rigidity of constant heat flow.
\newblock {\em Calculus of Variations and Partial Differential Equations},
  57(6):156, 2018.

\bibitem{schiff2013laplace}
J.~L. Schiff.
\newblock {\em The Laplace transform: theory and applications}.
\newblock Springer Science \& Business Media, 2013.

\bibitem{VM15}
V.~F. Vil'danova and F.~K. Mukminov.
\newblock T\"acklind uniqueness classes for heat equation on noncompact
  {R}iemannian manifolds.
\newblock {\em Ufa Math. J.}, 7(2):55--63, 2015.

\bibitem{Wa19}
B.~Wang.
\newblock Stability of catenoids and helicoids in hyperbolic space.
\newblock {\em Asian J. Math.}, 23(2):349--367, 2019.

\bibitem{Yau78}
S.~T. Yau.
\newblock On the heat kernel of a complete {R}iemannian manifold.
\newblock {\em J. Math. Pures Appl. (9)}, 57(2):191--201, 1978.

\bibitem{Za87}
L.~Zalcman.
\newblock Some inverse problems of potential theory.
\newblock {\em Contemp. Math.}, 63:337--350, 1987.

\end{thebibliography}
\end{document}